\documentclass[11pt,a4paper]{amsart}

\usepackage[T1]{fontenc}
\usepackage{graphicx}
\usepackage{epstopdf}
\usepackage[hyphens]{url}
\usepackage{microtype}
\usepackage[margin=2.5cm]{geometry}
\usepackage{amsmath}
\usepackage{mathtools}            
\usepackage{amsfonts}
\usepackage{latexsym}
\usepackage{amssymb}
\usepackage{stmaryrd}              
\usepackage{tikz}
\usetikzlibrary{automata,positioning}
\usetikzlibrary{cd}                  
\usepackage{hyperref}
\usepackage{cleveref}
\usepackage{enumerate}

\newtheorem{definition}{Definition}[section]
\newtheorem{lemma}[definition]{Lemma}
\newtheorem{proposition}[definition]{Proposition}
\newtheorem{question}[definition]{Question}

\newtheorem{theorem}[definition]{Theorem}
\newtheorem{remark}[definition]{Remark}

\theoremstyle{definition}
\newtheorem{example}[definition]{Example}

\newcommand{\N}{\mathbb{N}}
\newcommand{\Z}{\mathbb{Z}}

\newcommand{\D}{\mathbb{D}}
\newcommand{\0}{\mathtt{0}}
\newcommand{\1}{\mathtt{1}}
\newcommand{\2}{\mathtt{2}}

\newcommand{\ttd}{\mathtt{d}}
\newcommand{\Zp}{{\mathbb{Z}_{\geq 0}}}
\newcommand{\Zn}{{\mathbb{Z}_{< 0}}}
\newcommand{\Wmin}{{\mathtt{W}_{\min}}}
\newcommand{\Wmax}{{\mathtt{W}_{\max}}}

\newcommand{\Acal}{\mathcal{A}}

\newcommand{\Dcal}{\mathcal{D}}
\newcommand{\Fcal}{\mathcal{F}}

\newcommand{\Lcal}{\mathcal{L}}

\newcommand{\bn}{{\boldsymbol{n}}}

\DeclareMathOperator{\rep}{rep}
\DeclareMathOperator{\tail}{tail}
\DeclareMathOperator{\pad}{pad}

\DeclareMathOperator{\val}{val}

\DeclareMathOperator{\Per}{Per}

\newcommand{\valTwoC}{\val_{2c}}

\newcommand{\valFc}{\val_{\Fcal c}}

\newcommand{\repFc}{\rep_{\Fcal c}}

\begin{document}

\title{Dumont--Thomas complement numeration systems for $\Z$}

\author[S.~Labb\'e]{S\'ebastien Labb\'e}
\address[S.~Labb\'e]{Univ. Bordeaux, CNRS, Bordeaux INP, LaBRI, UMR 5800, F-33400 Talence, France}
\email{sebastien.labbe@labri.fr}
\urladdr{http://www.slabbe.org/}

\author[J.~Lep\v{s}ov\'a]{Jana Lep\v{s}ov\'a}
\address[J.~Lep\v{s}ov\'a]{FNSPE, CTU in Prague, Trojanova 13, 120 00 Praha, Czech Republic}
\email{jana.lepsova@labri.fr}

\makeatletter
\@namedef{subjclassname@2020}{\textup{2020} Mathematics Subject Classification}
\makeatother



\maketitle

\begin{abstract}
We extend the well-known Dumont--Thomas numeration systems to $\mathbb{Z}$ using an approach inspired by the two's complement numeration system. Integers in $\mathbb{Z}$ are canonically represented by a finite word (starting with $\mathtt{0}$ when nonnegative and with $\mathtt{1}$ when negative). The systems are based on two-sided periodic points of substitutions as opposed to the right-sided fixed points. For every periodic point of a~substitution, we construct an automaton which returns the letter at position $n\in\mathbb{Z}$ of the periodic point when fed with the representation of $n$ in the corresponding numeration system. The numeration system naturally extends to $\mathbb{Z}^d$. We give an equivalent characterization of the numeration system in terms of a total order on a regular language. Lastly, using particular periodic points, we recover the well-known two's complement numeration system and the Fibonacci analogue of the two's complement numeration system.
\end{abstract}


\section{Introduction}

On a finite size memory representing unsigned integers with base-10 digits,
incrementing by 1 the largest representable number 
gives
\[
    \begin{array}{r}
    \mathtt{9999999999999}\\
                \mathtt{+1}\\
    \hline
    \mathtt{0000000000000}\\
    \end{array}
\]
if we ignore the overflow error caused by the propagation of the carry beyond the memory limit.
Therefore, it makes sense to identify the number $\mathtt{999\cdots 9}$ with
the value $-1$ since adding one to it gives zero.
Likewise, 
$\mathtt{999\cdots 98}$ can be identified with the value $-2$,
$\mathtt{999\cdots 97}$ with the value $-3$, and so on, just like negative $p$-adic integers
\cite{zbMATH07214196}.
This numeration system is called ten's-complement.
As instructively explained by Knuth \cite[\S 4.1]{MR3077153}, the same can be
done in an arbitrary integer base $b\geq2$. When $b=2$, it is called the two's
complement numeration system. This system is still used nowadays to
represent signed integers in the architecture of modern processors 
\cite[\S 4.2.1]{intel-dev-manual-2023} due to its efficiency at performing
arithmetic operations. 

In this article, we show that the concept of complement numeration systems goes
beyond numeration systems in an integer base. 
The theory of numeration systems studies and describes the various ways of
representing numbers (integers, real numbers, Gaussian integers, etc.) by sequences of digits 
\cite{MR942576,MR992303,MR1411227,MR1463527,MR1491655,MR2742574,rigo_formal_2014}.
One of these ways 
gives rise to the numeration systems based on substitutions which were proposed by Dumont
and Thomas \cite{MR1020484}.
The Dumont--Thomas numeration system associated with a substitution provides 
a~canonical representation for every nonnegative integer.
It may also be used to represent real numbers in a certain interval.
It turns out there exists a natural complement version of the Dumont--Thomas
numeration systems allowing to represent all integers in $\Z$ and not only
those that are nonnegative.

\begin{figure}[h]
\begin{center}
\begin{tikzpicture}[scale=1.6, >=latex]
    \node                       (O) at (.5,1.6) {};
    \node[above,rectangle,draw] (a) at (.5,1) {$a$};
    \node[above,rectangle,draw] (b) at (0,0) {$b$};
    \node[above,rectangle,draw] (c) at (1,0) {$c$};
    \draw[->] (O) to (a);
    \draw[->,loop left] (a) to node[left] {\scriptsize $\0$} (a);
    \draw[->] (a) to node[fill=white,inner sep=2pt] {\scriptsize $\1$} (b);
    \draw[->,bend left=15] (a) to node[fill=white,inner sep=2pt] {\scriptsize $\2$} (c);
    \draw[->] (b) to node[fill=white,inner sep=2pt] {\scriptsize $\0$} (c);
    \draw[->,bend left=15] (c) to node[fill=white,inner sep=2pt] {\scriptsize $\0$} (a);
    \draw[->,loop right] (c) to node[right] {\scriptsize $\1$} (c);
\end{tikzpicture}
\end{center}
    \caption{The graph associated to the substitution $a\mapsto abc, b\mapsto c, c\mapsto ac$.}
    \label{fig:intro-graph}
\end{figure}

In practical terms, the Dumont--Thomas numeration system can be defined by the set of
finite paths in a directed graph starting from some fixed vertex. 
For example, consider the directed graph shown in Figure~\ref{fig:intro-graph}
with vertices $a$, $b$ and $c$ where the outgoing edges of every vertex are
labeled with consecutive nonnegative integers starting with zero.
The set of paths of fixed length starting with some chosen vertex can be unfolded into a tree,
see Figure~\ref{fig:intro-tree} (left). 
A path in the tree is uniquely identified with the sequence of labels of its
edges starting from the root. Among the set of paths of a given length ordered lexicographically,
the $n$-th one can be regarded as a representation of the nonnegative
integer $n$. Considering arbitrarily long finite paths starting from the initial vertex in the
directed graph, we obtain a canonical representation of all nonnegative integers after
removing leading zeros in their representation (assuming the initial vertex has a
loop labeled with $\mathtt{0}$); see
Figure~\ref{fig:intro-tree} (right).
\begin{figure}[h]
\begin{minipage}{0.4\textwidth}
\begin{center}
\begin{tikzpicture}[xscale=.75, yscale=1.5, >=latex]
    \node                 (O)   at (0,.5) {};
    \node[rectangle,draw] (00a) at (0,0)  {$a$};
    \node[rectangle,draw] (10a) at (0,-1) {$a$};
    \node[rectangle,draw] (11b) at (1,-1) {$b$};
    \node[rectangle,draw] (12c) at (2,-1) {$c$};
    \node[rectangle,draw] (20a) at (0,-2) {$a$};
    \node[rectangle,draw] (21b) at (1,-2) {$b$};
    \node[rectangle,draw] (22c) at (2,-2) {$c$};
    \node[rectangle,draw] (23c) at (3,-2) {$c$};
    \node[rectangle,draw] (24a) at (4,-2) {$a$};
    \node[rectangle,draw] (25c) at (5,-2) {$c$};
    \node at (0,-2.5) {$0$};
    \node at (1,-2.5) {$1$};
    \node at (2,-2.5) {$2$};
    \node at (3,-2.5) {$3$};
    \node at (4,-2.5) {$4$};
    \node at (5,-2.5) {$5$};
    \draw[->] (O) to (00a);
    \draw[->] (00a) -- node[fill=white,inner sep=2pt] {\scriptsize $\0$} (10a);
    \draw[->] (00a) -- node[fill=white,inner sep=2pt] {\scriptsize $\1$} (11b);
    \draw[->] (00a) -- node[fill=white,inner sep=2pt] {\scriptsize $\2$} (12c);
    \draw[->] (10a) -- node[fill=white,inner sep=2pt] {\scriptsize $\0$} (20a);
    \draw[->] (10a) -- node[fill=white,inner sep=2pt] {\scriptsize $\1$} (21b);
    \draw[->] (10a) -- node[fill=white,inner sep=2pt] {\scriptsize $\2$} (22c);
    \draw[->] (11b) -- node[fill=white,inner sep=2pt] {\scriptsize $\0$} (23c);
    \draw[->] (12c) -- node[fill=white,inner sep=2pt] {\scriptsize $\0$} (24a);
    \draw[->] (12c) -- node[fill=white,inner sep=2pt] {\scriptsize $\1$} (25c);
\end{tikzpicture}
    \end{center}
    \end{minipage}
    \begin{minipage}{0.4\textwidth}
        \begin{center}
        \begin{tabular}{|c|c|c|}
                       $n$ & path & representation of $n$\\ \hline
                        0  & $\mathtt{00}$          & $\varepsilon$          \\
                        1  & $\mathtt{01}$          & $\mathtt{1}$           \\
                        2  & $\mathtt{02}$          & $\mathtt{2}$           \\
                        3  & $\mathtt{10}$          & $\mathtt{10}$          \\
                        4  & $\mathtt{20}$          & $\mathtt{20}$          \\
                        5  & $\mathtt{21}$          & $\mathtt{21}$          \\
        \end{tabular}
        \end{center}
    \end{minipage}
    \caption{The set of paths starting in state $a$ in the directed graph
    provide a~canonical representation of the nonnegative integers after
    removing leading zeroes.
    }
    \label{fig:intro-tree}
\end{figure}
We refer to such a numeration system as to the Dumont--Thomas numeration system for~$\N$
associated with the substitution $\mu:a\mapsto abc, b\mapsto c, c\mapsto ac$. 
The directed graph shown in Figure~\ref{fig:intro-tree} 
(as well as the automaton shown in Figure~\ref{fig:intro-graph})
is derived from the substitution $\mu$ following a well-known construction for automatic sequences \cite{MR1997038}: 
$\alpha\xrightarrow{i}\beta$ is an edge of the graph if and only if
$\beta$ is the $i$-th letter of the image of the letter $\alpha$, for every integer $i$ such that $0\leq i < \ell$ where $\ell$ is the length of the image of the letter $\alpha$.
Among other properties, this numeration system gives a direct description of
the right-infinite fixed point $\textbf{t} = \mu(\textbf{t})=abccacacabc\ldots$
of the substitution $\mu$ as an automatic sequence.

The Dumont--Thomas numeration systems were later explained using the
so-called prefix-suffix automata associated with primitive substitutions
\cite{MR1879663} by considering the cylinders of finite length words
\cite[Corollary 6.2]{MR1879663}.
The main motivation of Canterini and Siegel was
to prove that every dynamical system
generated by a substitution of Pisot type on $d$ letters admits 
a~minimal translation on the torus $\mathbb{T}^{d-1}$
as a
topological factor 
\cite{MR1852097}. As a consequence, 
they obtained a numeration system representing the
elements of $\mathbb{T}^{d-1}$ by infinite paths in a prefix-suffix automaton;
see \cite{MR1970385}.

In more generality, every regular language over a totally ordered alphabet
leads to what is called an \emph{abstract numeration system},
which may be used to represent nonnegative integers \cite{MR1799066}
or real numbers in an interval \cite{zbMATH01748010},
see also \cite[\S 7]{MR1905123}, \cite[\S 4]{MR2290784}, 
\cite{MR2282862}, \cite[\S 3]{MR2742574}, \cite{MR2799274}.

Recently, a~very general framework was proposed to 
extend the Dumont--Thomas numeration systems to all integers
based on the notion of coding prescription, which allows the image of
letters to be scattered words of nonconsecutive letters \cite{MR3805464}.
Another recent article extending these numeration systems concerns also
the $\beta$-numeration of real numbers in an interval \cite{MR4143682}.

\textbf{In this contribution.}
The extension of the Dumont--Thomas
numeration systems to all integers in $\Z$ that we propose
is inspired by integer base complement numeration systems; see
Definition~\ref{def:rep-not-prolongable}.
It is derived from the two-sided periodic points of substitutions
as opposed to the right-infinite fixed points. 
In a Dumont--Thomas complement numeration system, the representations of nonnegative integers
start with the digit $\0$ whereas the 
representations of negative integers start with the digit $\1$.
The proofs provided here follow as much as possible the approach originally
proposed by Dumont and Thomas \cite{MR1020484}.

The main results of this contribution are Theorem~\ref{th:periodic_point_automatic},
where we prove that two-sided periodic points of substitutions are automatic sequences
with respect to the Dumont--Thomas complement numeration systems for $\Z$,
and Theorem~\ref{thm:dumont-thomas-characterization-increasing},
where we characterize these numeration systems
by means 
of a total order
on the language recognized by an automaton.
Finally, we show that the well-known two's complement numeration system can be
constructed as a Dumont--Thomas complement numeration system for $\Z$ 
(Proposition~\ref{prop:rep_equality_2c})
and similarly for the
Fibonacci analogue of the two's complement numeration system
(Proposition~\ref{prop:rep_equality}).

Also, we extend the Dumont--Thomas complement numeration systems to $\Z^d$; see
Definition~\ref{def:num-sys-Zd}.
The need for extending the theory of numeration systems based on substitutions from $\N$
to $\Z$ and to $\Z^d$ for $d\geq 2$ was motivated by the study of aperiodic Wang tilings of the
plane. In \cite{MR4364231}, configurations in a particular aperiodic Wang shift
based on 16 Wang tiles were described by an automaton derived from a
two-dimensional substitution. The automaton takes as input the representation of a
position in $\Z^2$ using a Fibonacci analogue of the two's complement
numeration system and outputs the index of the Wang tile to place at this
position. 
This example belongs to a family of the Dumont--Thomas complement numeration systems for $\Z^2$.

The authors believe further extensions beyond Dumont--Thomas based on a single
substitution can be expected including $S$-adic sequences \cite{zbMATH01019492}.
For instance, 
in a Bratteli--Vershik diagram \cite{MR1194074,MR1322556,MR3791491},
one may think of the maximal path in the diagram as a representation of $-1$
and the minimal path as a representation of $0$.
The representation of the other negative and nonnegative integers can be deduced from 
the order of a Bratteli--Vershik
diagram and its natural successor map.

\subsection*{Structure of the article}
Preliminaries and notation are presented in Section~\ref{sec:prelim}.
Section~\ref{sec:dumont-thomas-N}
recalls numeration systems for $\N$ defined by Dumont and Thomas
and presents some extensions of their results.
In Section~\ref{sec:dumont-thomas-Z}, 
we extend a theorem of Dumont and Thomas to the
right-infinite and left-infinite periodic
points of substitutions.
We use it to define numeration systems for $\Z$
based on the two-sided periodic points of substitutions.
In Section~\ref{sec:examples}, we show some examples.
In Section~\ref{sec:periodic-as-automatic}, we describe periodic points of substitutions
as automatic sequences.
In Section~\ref{sec:NS-for-Zd}, 
we show how to extend the Dumont--Thomas numeration systems to $\Z^d$.
In Section~\ref{sec:total-order}, we present a total order
on $\{\0,\1\}\odot\Dcal^*$, where $\Dcal$ is some alphabet of integers
and $\odot$ is the concatenation of words within the monoid $\Dcal^*$.
We characterize the Dumont--Thomas complement numeration systems for $\Z$ with
respect to this total order.  In Section~\ref{sec:known-complement-NS}, we show
that the
well-known two's complement numeration system is an instance of a
Dumont--Thomas numeration system for $\Z$ and similarly for the Fibonacci
analogue of the two's complement numeration system.


\section{Preliminaries}\label{sec:prelim}

An \emph{alphabet} $A$ is a finite set and its elements $a\in A$
are called $\emph{letters}$. 
A \emph{finite word} $u = u_0 u_1 \cdots u_{n-1}$ is a concatenation of letters 
$u_i \in A$ for every $i\in \{0,1,\dots,n-1\}$ and $|u|$ denotes its \emph{length}.
When it is more convenient, we denote the $i$-th letter of $u$
by $u[i]$ instead of $u_i$.
The \emph{empty word} is denoted by $\varepsilon$.
The set of all finite words over the alphabet $A$ is denoted by $A^*$ 
and the set of all nonempty words over the alphabet $A$ is denoted by
$A^+ = A^*\setminus \{\varepsilon\}$.
We define the concatenation $\odot$ as the following binary operation:
\[
    \odot:A^*\times A^*\to A^*, u\odot v\mapsto uv.
\]
The set $A^*$ with the concatenation as operation
forms a monoid with $\varepsilon$ as the neutral element.

A \emph{morphism} over $A$ is a map $\eta:A^* \to A^*$ 
such that $\eta(u\odot v) = \eta(u)\odot\eta(v)$  for all words $u, v \in A^*$.
A \emph{substitution} $\eta:A^* \to A^*$ is a morphism such that 
$\eta(a)\in A^+$ is nonempty
for every $a\in A$
and there exists $a\in A$ such that $a$ is
\emph{growing}, that is,
$\lim_{k\to +\infty} |\eta^k(a)| = +\infty$.
A~morphism $\eta$ is said \emph{primitive} if 
there exists $k\in\N$ such that
for every $a,b\in A$ 
the letter $a$ appears in $\eta^k(b)$.
A~morphism $\eta$ is said \emph{$d$-uniform} for some nonnegative integer $d$
if $|\eta(a)|=d$ for every letter $a\in A$.

We call $u_0 u_1 u_2 \cdots \in A^\Zp$ a \emph{right-infinite word}
and $ \cdots  u_{-3} u_{-2} u_{-1} \in A^{\Zn}$ a \emph{left-infinite word}.
We call $u\in A^{\Z}$ a \emph{two-sided word} and
we separate by a vertical bar its elements $u_{-1}$ and $u_0$ to indicate the origin, i.e.,
$u = \cdots u_{-3} u_{-2} u_{-1} | u_{0} u_1 u_2 \cdots$.

Substitutions can be applied naturally to two-sided words $u\in A^{\Z}$ by setting
\[
\eta(\dots u_{-3} u_{-2} u_{-1} | u_0 u_1 u_2 \cdots) 
= \cdots\eta(u_{-3}) \eta(u_{-2}) \eta(u_{-1}) |\eta(u_0) \eta(u_1) \eta(u_2) \cdots.
\]
Let $\D\in\{\Z,\Zp,\Zn\}$.
A word $u\in A^\D$ is called a \emph{periodic point} of the substitution~$\eta$
if there exists an integer $p\geq 1$ such that $\eta^p(u)=u$,
and in this case, $p$ is called a \emph{period} of the periodic point.
The minimum integer $p\geq1$ such that $\eta^p(u)=u$ is called \emph{the} period of $u$.
A periodic point with period $p=1$ is called a \emph{fixed point} of $\eta$.
The set of periodic points of $\eta$ is denoted by
$\Per_\D(\eta)=\{u\in A^\D \mid \eta^p(u)=u\text{ for some }p\geq1\}$.
Since we are mostly interested in two-sided words in this contribution,
we omit the domain when $\D=\Z$ and we write
$\Per(\eta)=\Per_\Z(\eta)$.

If $u\in\Per(\eta)$ is a two-sided periodic point of a substitution $\eta$,
then we say that the pair of letters $u_{-1}|u_0$ is the \emph{seed} of $u$, 
see \cite[\S4.1]{MR3136260}.
If the seed letters of a two-sided periodic point are growing, then the
periodic point is defined entirely by its seed.
More precisely,
$u = \lim_{k\to +\infty} \eta^{pk}(u_{-1})|\eta^{pk}(u_0)$,
where $p$ is a period of $u$.


Let $u=u_0 u_1 \cdots\in\Per_\N(\eta)$ with a seed $u_0=a$.
The previous terminology is inspired by \cite{MR1879663},
where a prefix-suffix automaton is associated with $\eta$.
However, for our goal an automaton associated with $\eta$
as in \cite{MR2742574} is sufficient.
Let $\Dcal$ denote the alphabet
$\Dcal=\left\{\0,...,\max_{c\in A}|\eta(c)|-1\right\}$
whose elements are integers.
The set $\Dcal^*$ is a monoid for the operation $\odot$ of concatenation.
The deterministic finite automaton with output (DFAO)
associated to the substitution $\eta$ and letter $a$ is
the 5-tuple
$\Acal_{\eta,a} = (A,\Dcal,\delta,a,A)$,
where the transition function $\delta:A\times \Dcal \rightarrow A$ 
is a partial function 
such that $\delta(b,i)=c$ 
if and only if 
$c=w_i$ and $\eta(b)=w_0\dots w_{|\eta(b)|-1}$. 
The transition function $\delta$ is naturally extended to $A\times\Dcal^*$
by $\delta(b,\varepsilon) = b$ for every $b\in A$,
and, for every $b\in A$, $i\in\Dcal$ and $w\in \Dcal^*$,
$\delta(b,i\odot w) = \delta(\delta(b,i),w)$.
For some state $b \in A$ and word $w\in \Dcal^*$, we denote
$\Acal_{\eta,a}(b,w) = \delta(b,w)$. In particular, we denote
$\Acal_{\eta,a}(w) = \delta(a,w)$.
We let  $\Lcal(\Acal_{\eta,a})$ denote the words
accepted by the automaton $\Acal_{\eta,a}$
and for $q\in\N$ we denote $\Lcal_q(\Acal_{\eta,a})$
the set of words $w\in\Lcal(\Acal_{\eta,a})$ such that $|w|=q$.

\section{Dumont--Thomas numeration system for $\N$}\label{sec:dumont-thomas-N}

In this section, we recall Dumont--Thomas numeration system for $\N$, which was 
based on  substitutions having a right-infinite fixed point \cite{MR1020484}. 
It uses the definition of admissible sequences.

\begin{definition}[admissible sequence]
{\rm\cite{MR1020484}}
Let $\eta:A^*\to A^*$ be a substitution. 
Let $a\in A$ be a letter, $k$ an integer and, for each integer $i$, $0\leq i\leq k$,
$(m_i,a_i)$ be an element of $A^*\times A$. We say that the finite sequence 
$(m_i,a_i)_{i=0,\dots,k}$ is \emph{admissible with respect to $\eta$} if and only if, for all $i$, 
$1\leq i\leq k$,
$m_{i-1}a_{i-1}$ is a prefix of $\eta(a_i)$. We say that this sequence is
\emph{$a$-admissible with respect to $\eta$} if it is admissible with respect to $\eta$ and, moreover, $m_ka_k$ is a prefix of
$\eta(a)$.
\end{definition}

As done in \cite{MR1020484}, when the substitution is clear from the
context, we write that a sequence is \emph{admissible} or \emph{$a$-admissible} without
specifying the substitution.


Dumont and Thomas proved the following result, which we rewrite in our notation.

\begin{theorem}\label{thm:dumont-thomas}
    {\rm\cite[Theorem 1.5]{MR1020484}}
    Let $a\in A$ and let $\eta:A^*\to A^*$ be a substitution. 
    Let $u=\eta(u)$ be a right-infinite fixed point of $\eta$ with growing seed $u_0 = a$.
    For every integer $n\geq1$,
    there exists a unique integer $k=k(n)$ and a unique sequence
    $(m_i,a_i)_{i=0,\dots,k}$ such that
    \begin{itemize}
        \item this sequence is $a$-admissible and $m_k\neq\varepsilon$, 
        \item $u_0u_1\cdots u_{n-1} = \eta^k(m_k) \eta^{k-1}(m_{k-1}) \cdots \eta^0(m_0)$.
    \end{itemize}
\end{theorem}

The proof of the above theorem was based on the following lemmas,
which we cite here as we need them in what follows.

\begin{lemma}\label{lem:admissible-sequence-sum}
    {\rm\cite[Lemma 1.1]{MR1020484}}
    Let $\eta:A^*\to A^*$ be a substitution and
    $k\geq0$ be an integer. If 
    $(m_i,a_i)_{i=0,\dots,k}$ is an admissible sequence, then 
    \[
        \sum_{j=0}^k |\eta^j(m_j)| < |\eta^k(m_k a_k)|.
    \]
\end{lemma}


\begin{lemma}\label{lem:unique_admissible_sequence}
    {\rm\cite[Lemma 1.3]{MR1020484}}
    Let $\eta:A^*\to A^*$ be a substitution and
    $k\geq0$ be an integer. Let $b\in A$,
    $(m_i,a_i)_{i=0,\dots,k}$ and
    $(m'_i,a'_i)_{i=0,\dots,k}$ be two $b$-admissible sequences
    and $n$ be an integer such that
    \[
        n = \sum_{j=0}^k |\eta^j(m_j)| = \sum_{j=0}^{k} |\eta^j(m'_j)|.
    \]
    Then for every $i$, $0\leq i \leq k$, we have $(m_i,a_i) = (m'_i, a'_i)$.
\end{lemma}

\begin{lemma}\label{lem:prefix_recursion}
    {\rm\cite[Lemma 1.4]{MR1020484}}
    Let $\eta\colon A^*\to A^*$ be a substitution.
    Let $\ell\geq 1$ be an integer, $a\in A$ a letter and $m\in A^*$
    a proper prefix of the word $\eta^\ell(a)$.
    Then there exist $(m',a')\in A^*\times A$ and $m''\in A^*$
    such that $m'a'$ is a prefix of $\eta(a)$,
    $m''$ is a proper prefix of $\eta^{\ell-1}(a')$
    and $m = \eta^{\ell-1}(m') m''$.
\end{lemma}

\subsection{Some extensions of Dumont--Thomas results}

In this subsection, we propose some extensions of Dumont--Thomas lemmas.
Firstly, we observe that admissible sequences are related to automata as follows.

\begin{lemma}\label{lemma:admissible-feed-into-automaton}
    Let $\eta:A^*\to A^*$ be a substitution,
    $k\geq1$ be an integer and $x\in A$.
    If $(m_i,a_i)_{i=0,\dots,k-1}$ is an $x$-admissible sequence,
    then 
    \[
        a_i=\Acal_{\eta,x}(|m_{k-1}|\odot|m_{k-2}|\odot\ldots\odot|m_i|) 
        \qquad \text{ for every } i=0,\dots,k-1.
    \]
\end{lemma}

\begin{remark}
    The notation $(\odot)$ in the above equation and the proof that follows
    stands for the concatenation of words within the monoid $\Dcal^*$. 
    Since the elements of $\Dcal$ are integers,
    we write this notation explicitly to avoid misinterpreting it with the
    multiplication of integers.
\end{remark}

\begin{proof}
    The proof is carried out by induction on $i$.
    If $i=k-1$, then
        $a_i=a_{k-1}=\eta(x)[|m_{k-1}|]=\Acal_{\eta,x}(|m_{k-1}|)$.
    If $i<k-1$, then
    \begin{align*}
        a_i 
        &= \eta(a_{i+1})[|m_i|]
        = \eta\left(\Acal_{\eta,x}(|m_{k-1}|\odot\ldots\odot|m_{i+1}|)\right)[|m_i|]\\
        &= \Acal_{\eta,x}(|m_{k-1}|\odot\ldots\odot|m_{i+1}|\odot|m_i|).\qedhere
    \end{align*}
\end{proof}

\begin{lemma}\label{lem:automaton-language-is-admissible}
    Let $\eta:A^*\to A^*$ be a substitution,
    $k\geq1$ be an integer and $x\in A$.
    If $v_{k-1}v_{k-2}\cdots v_0\in\Lcal(\Acal_{\eta,x})$, then
    there exists
    an $x$-admissible sequence
    $(m_i,a_i)_{i=0,\dots,k-1}$ such that
    $|m_i|=v_i$ for every $i=0,\dots,k-1$.
\end{lemma}

\begin{proof}
    We carry out the proof by induction on $k$.
    If $k = 1$, then $v_0 \in\Lcal(\Acal_{\eta,x})$ implies that $0\leq v_0 < |\eta(x)|$.
    Denote $m_0$ the proper prefix of $\eta(x)$ of length $v_0$ and denote $a_0\in A$
    so that $m_0 a_0$ is a prefix of  $\eta(x)$. 
    The length-1 sequence $(m_i,a_i)_{i=0}$ is $x$-admissible
    and satisfies the condition that $|m_0|=v_0$.

    Induction hypothesis: 
    for some integer $k\geq1$ it holds that for every word 
    $w_{k-1}w_{k-2}\cdots w_0\in\Lcal(\Acal_{\eta,x})$ of length $k$,
    there exists
    an $x$-admissible sequence
    $(m_i,a_i)_{i=0,\dots,k-1}$ such that
    $|m_i|=w_i$ for every $i=0,\dots,k-1$.
    Let
    $v_{k}v_{k-1}\cdots v_0\in\Lcal(\Acal_{\eta,x})$.
    Then from the induction hypothesis applied on 
    $v_{k}v_{k-1}\cdots v_1\in\Lcal(\Acal_{\eta,x})$, which is of length $k$,
    we have an $x$-admissible sequence
    $(m_i,a_i)_{i=1,\dots,k}$ such that
    $|m_i|=v_i$ for every $i=1,\dots,k$.
    We have from Lemma~\ref{lemma:admissible-feed-into-automaton}
    that $a_1 = \Acal_{\eta,x}(v_k v_{k-1} \cdots v_1)$ and
    we have
    from the definition of the automaton 
    $\Acal_{\eta,x}$ that $v_0 < |\eta(a_1)|$. 
    Denote $m_0$ the proper prefix of $\eta(a_1)$ of length $v_0$ and denote $a_0\in A$
    so that $m_0 a_0$ is a prefix of  $\eta(a_1)$. 
    Then $|m_0| = v_0$ and 
    $(m_i,a_i)_{i=0,\dots,k}$ is an $x$-admissible sequence.
\end{proof}

Lemma~\ref{lem:prefix_recursion} can be used to construct
an admissible sequence from a prefix of the image of a~letter
under the $p$-th power of a substitution.

\begin{lemma}\label{lem:when-m-prefix-of-eta-p-a}
    Let $\eta:A^*\to A^*$ be a substitution and
    $p\geq 1$ be an integer.
    If $m\in A^*$ and $x\in A$ are such that
    $m$ is a proper prefix of $\eta^p(x)$, then
    there exists a unique $x$-admissible sequence
    $(m_i,a_i)_{i=0,\dots,p-1}$ such that
    \begin{equation}\label{eq:m-as-product}
        |m| = \textstyle\sum_{j=0}^{p-1} |\eta^j(m_j)|.
    \end{equation}
    Moreover, $m = \eta^{p-1}(m_{p-1}) \eta^{p-2}(m_{p-2}) \cdots \eta^0(m_0)$.
\end{lemma}

\begin{proof}
    (Uniqueness)
    Let
    $(m_i,a_i)_{i=0,\dots,p-1}$ and
    $(m'_i,a'_i)_{i=0,\dots,p-1}$ be two $x$-admissible sequences satisfying the hypothesis.
    Then
    \[
            \textstyle\sum_{j=0}^{p-1} |\eta^j(m_j)| 
          = |m|
          = \textstyle\sum_{j=0}^{p-1} |\eta^j(m'_j)|.
    \]
    By Lemma~\ref{lem:unique_admissible_sequence},
    $(m_i,a_i)_{i=0,\dots,p-1}=(m'_i,a'_i)_{i=0,\dots,p-1}$.

    (Existence)
    We carry out the proof by induction on $p$.
    If $p=1$, then $m$ is a proper prefix of $\eta(x)$.
    Let $m_0=m$ and $a_0\in A$ be such that $m a_0$ is a prefix of $\eta(x)$.
    The length-1 sequence $(m_i,a_i)_{i=0}$ is $x$-admissible
    and satisfies the condition that $m=\eta^0(m_0)$.

    Now let $m\in A^*$ and $x\in A$ be such that
    $m$ is a proper prefix of $\eta^{p+1}(x)$.
    From Lemma~\ref{lem:prefix_recursion}, 
    there exist $(m_p,a_p)\in A^*\times A$ and $m''\in A^*$
    such that $m_pa_p$ is a prefix of $\eta(x)$,
    $m''$ is a proper prefix of $\eta^{p}(a_p)$
    and $m = \eta^{p}(m_p) m''$.
    By the induction hypothesis,
    there exists an $a_p$-admissible sequence
    $(m_i,a_i)_{i=0,\dots,p-1}$ such that
    \[
        m'' = \eta^{p-1}(m_{p-1}) \eta^{p-2}(m_{p-2}) \cdots \eta^0(m_0).
    \]
    Therefore, 
    \[
        m = \eta^{p}(m_p) m''
          = \eta^{p}(m_p)\eta^{p-1}(m_{p-1}) \eta^{p-2}(m_{p-2}) \cdots \eta^0(m_0).
    \]
    The extended sequence $(m_i,a_i)_{i=0,\dots,p}$ is $x$-admissible
    since $m_{p-1}a_{p-1}$ is a prefix of $\eta(a_p)$
    and $m_pa_p$ is a prefix of $\eta(x)$.
\end{proof}

Let $\eta:A^*\to A^*$ be a substitution and
$\Dcal=\left\{\0,...,\max_{c\in A}|\eta(c)|-1\right\}$.
Lemma~\ref{lem:when-m-prefix-of-eta-p-a} allows us to define
a map for every integer $p\geq 1$ and $x\in A$ as follows
	\[
	\begin{array}{rccl}
        \tail_{\eta,p,x}:&\{\0,\1,\dots,|\eta^p(x)|-1\} & \to & \Dcal^p\\
		&n & \mapsto & |m_{p-1}|\odot|m_{p-2}|\odot\ldots\odot|m_0|,
	\end{array}
	\]
where 
$(m_i,a_i)_{i=0,\dots,p-1}$ is the unique $x$-admissible sequence
satisfying Equation~\eqref{eq:m-as-product}
with $m$ being the prefix of length $n$ of $\eta^p(x)$.
The map $\tail_{\eta,p,x}$ will be used in Section~\ref{sec:dumont-thomas-Z}.

\begin{example}\label{ex:Tribo}
    Consider the Tribonacci substitution $\psi_{T}: a\mapsto ab, b\mapsto ac,
    c\mapsto a$ \cite{Ra82}.
    The successive images of $a$ under the substitution $\psi_T$ are illustrated below
    in a tree.
\begin{center}
\begin{tikzpicture}[yscale=1.5]
    \node[above,rectangle,draw] (00a) at (0,0)  {$a$};
    \node[above,rectangle,draw] (10a) at (0,-1) {$a$};
    \node[above,rectangle,draw] (11b) at (1,-1) {$b$};
    \node[above,rectangle,draw] (20a) at (0,-2) {$a$};
    \node[above,rectangle,draw] (21b) at (1,-2) {$b$};
    \node[above,rectangle,draw] (22a) at (2,-2) {$a$};
    \node[above,rectangle,draw] (23c) at (3,-2) {$c$};
    \node[above,rectangle,draw] (30a) at (0,-3) {$a$};
    \node[above,rectangle,draw] (31b) at (1,-3) {$b$};
    \node[above,rectangle,draw] (32a) at (2,-3) {$a$};
    \node[above,rectangle,draw] (33c) at (3,-3) {$c$};
    \node[above,rectangle,draw] (34a) at (4,-3) {$a$};
    \node[above,rectangle,draw] (35b) at (5,-3) {$b$};
    \node[above,rectangle,draw] (36a) at (6,-3) {$a$};
    \node[above] (40a) at (0,-3.5) {$0$};
    \node[above] (41b) at (1,-3.5) {$1$};
    \node[above] (42a) at (2,-3.5) {$2$};
    \node[above] (43c) at (3,-3.5) {$3$};
    \node[above] (44a) at (4,-3.5) {$4$};
    \node[above] (45b) at (5,-3.5) {$5$};
    \node[above] (46a) at (6,-3.5) {$6$};
    \node[above] (47a) at (7,-3.5) {$\cdots$};

    \draw(-.5,-3.1) -- (7.5,-3.1);

    \draw[->,dashed] (00a) -- node[fill=white] {$\0$} (10a);
    \draw[->] (00a) -- node[fill=white] {$\1$} (11b);
    \draw[->,dashed] (10a) -- node[fill=white] {$\0$} (20a);
    \draw[->,dashed] (10a) -- node[fill=white] {$\1$} (21b);
    \draw[->] (11b) -- node[fill=white] {$\0$} (22a);
    \draw[->] (11b) -- node[fill=white] {$\1$} (23c);
    \draw[->,dashed] (20a) -- node[fill=white] {$\0$} (30a);
    \draw[->,dashed] (20a) -- node[fill=white] {$\1$} (31b);
    \draw[->,dashed] (21b) -- node[fill=white] {$\0$} (32a);
    \draw[->,dashed] (21b) -- node[fill=white] {$\1$} (33c);
    \draw[->] (22a) -- node[fill=white] {$\0$} (34a);
    \draw[->] (22a) -- node[fill=white] {$\1$} (35b);
    \draw[->] (23c) -- node[fill=white] {$\0$} (36a);
\end{tikzpicture}
\end{center}
The path from the root of the tree to a node of depth $p$ at $x$-position $n\in\N$ is
labeled by $\tail_{\psi_T,p,a}(n)$. Their values are illustrated in the following table.
    \[
\begin{array}{l|c|c|c}
n & \tail_{\psi_T,1,a}(n) & \tail_{\psi_T,2,a}(n) & \tail_{\psi_T,3,a}(n)\\
\hline
0 & \0           & \0\0         & \0\0\0\\
1 & \1           & \0\1         & \0\0\1\\
2 &              & \1\0         & \0\1\0\\
3 &              & \1\1         & \0\1\1\\
4 &              &              & \1\0\0\\
5 &              &              & \1\0\1\\
6 &              &              & \1\1\0\\
\end{array}
\]
\end{example}

\begin{lemma}\label{lem:automaton_q_r}
    Let $\eta:A^*\to A^*$ be a substitution and $p\geq 1$ be an integer.
    Let $x\in A$. Then for every $\ell\in\{\0,\1,\dots, |\eta^p(x)|-1\}$ we have
    \begin{equation*}
        \eta^p(x)[\ell] = \Acal_{\eta,x}(\tail_{\eta,p, x}(\ell)).
    \end{equation*}
\end{lemma}

\begin{proof}
    Let $m$ be the prefix of $\eta^p(x)$ of length $\ell$. 
    From Lemma~\ref{lem:when-m-prefix-of-eta-p-a},
    there exists a unique $x$-admissible sequence
    $(m_i,a_i)_{i=0,\dots,p-1}$ such that
    \[
        m = \eta^{p-1}(m_{p-1}) \eta^{p-2}(m_{p-2}) \cdots \eta^0(m_0).
    \]
    The word $ma_0$ is a prefix of $\eta^p(x)$, thus 
    $\eta^p(x)[\ell]=a_0$.
    From Lemma~\ref{lemma:admissible-feed-into-automaton},
    \[
        \eta^p(x)[\ell]=
        a_0=\Acal_{\eta,x}(|m_{p-1}|\odot|m_{p-2}|\odot\ldots\odot|m_0|)
           =\Acal_{\eta,x}(\tail_{\eta,p, x}(\ell)).\qedhere
    \]
\end{proof}

In the next lemma, we consider the total order
$(\Dcal^*,<_{lex})$, where $u<_{lex}v$ means that 
$u$ is lexicographically less than $v$.
Recall that given a totally ordered set $(\Dcal,<)$, and two words
$u,v\in\Dcal^*$ such that $v$ is nonempty, then one has 
that $u$ is lexicographically less than $v$, if $u$ is a~proper prefix of $v$, or
there exist words $r,s,t\in\Dcal^*$ and letters $a,b\in\Dcal$ such that $u =
ras$ and $v = rbt$ with $a < b$.

\begin{lemma}\label{lem:lexico-order-on-admissible-seq}
    Let $\eta:A^*\to A^*$ be a substitution and
    $p\geq 1$ be an integer.
    Let $n,n'\in\{\0,\1,\dots,|\eta^p(x)|-1\}$.
    Then
    \begin{enumerate}[(i)]
        \item $n=n'$ if and only if
            $\tail_{\eta,p,x}(n) =\tail_{\eta,p,x}(n')$,
        \item $n<n'$ if and only if
            $\tail_{\eta,p,x}(n) <_{lex}\tail_{\eta,p,x}(n')$.
    \end{enumerate}
\end{lemma}

\begin{proof}
    Let
    $(m_i,a_i)_{i=0,\dots,p-1}$
    and
    $(m'_i,a'_i)_{i=0,\dots,p-1}$
    be two $x$-admissible sequences such that
    $n=\textstyle\sum_{j=0}^{p-1} |\eta^j(m_j)|$
    and $n'=\textstyle\sum_{j=0}^{p-1} |\eta^j(m'_j)|$.
    Thus
            $\tail_{\eta,p,x}(n)
            = |m_{p-1}|\odot|m_{p-2}|\odot\ldots\odot|m_0|$
    and     $\tail_{\eta,p,x}(n')
            = |m'_{p-1}|\odot|m'_{p-2}|\odot\ldots\odot|m'_0|$.

    (i)
    If $n=n'$, then 
    $\tail_{\eta,p,x}(n) =\tail_{\eta,p,x}(n')$.
    Conversely, if $\tail_{\eta,p,x}(n) =\tail_{\eta,p,x}(n')$, then
    $m_{p-1}a_{p-1}=m'_{p-1}a'_{p-1}$ since both are prefixes
    of the same length of $\eta(x)$. Thus $m_{p-1}=m'_{p-1}$ and
    $a_{p-1}=a'_{p-1}$.
    Similarly, $m_{p-2}a_{p-2}=m'_{p-2}a'_{p-2}$ since both are prefixes
    of the same length of $\eta(a_{p-1})$. Thus $m_{p-2}=m'_{p-2}$ and
    $a_{p-2}=a'_{p-2}$. By induction, we obtain
    $(m_i,a_i)_{i=0,\dots,p-1}=(m'_i,a'_i)_{i=0,\dots,p-1}$.
    Thus $n=\textstyle\sum_{j=0}^{p-1} |\eta^j(m_j)|
    =\textstyle\sum_{j=0}^{p-1} |\eta^j(m'_j)|=n'$.

    (ii)
    Suppose that 
        $|m_{p-1}|\odot|m_{p-2}|\odot\ldots\odot|m_0|
        <_{lex} |m'_{p-1}|\odot|m'_{p-2}|\odot\ldots\odot|m'_0|$.
    Then there exists an integer $\ell$
    such that $0\leq \ell\leq p-1$,
    $|m_{j}|=|m'_j|$ for every integer $j$ such that $\ell<j\leq p-1$
    and $|m_\ell|<|m'_\ell|$.
    Since $|m_{p-1}|=|m'_{p-1}|$ and $m_{p-1}a_{p-1}$ and $m'_{p-1}a'_{p-1}$ are prefixes of $\eta(x)$
    we have that $m_{p-1}=m'_{p-1}$ and $a_{p-1}=a'_{p-1}$.
    Similarly, we have $m_j=m'_j$ and $a_j=a'_j$ for every $j$ such that $\ell<j\leq p-1$.
    Thus $m_\ell a_\ell$ and $m'_\ell a'_\ell$ must both be prefixes of 
    the image under $\eta$ of the same letter. 
    This letter is $x$ if $\ell=p-1$ or otherwise is $a_{\ell+1}=a'_{\ell+1}$.
    Since $|m_\ell|<|m'_\ell|$, we have that $m_\ell a_\ell$ is a prefix
    of $m'_\ell$.
    Using Lemma~\ref{lem:admissible-sequence-sum}, we have
    \begin{align*}
        n-n' 
        &= \textstyle\sum_{j=0}^{p-1} |\eta^j(m_j)|
         - \textstyle\sum_{j=0}^{p-1} |\eta^j(m'_j)|\\
        &= \textstyle\sum_{j=0}^{\ell} |\eta^j(m_j)|
         - \textstyle\sum_{j=0}^{\ell} |\eta^j(m'_j)|
         \leq |\eta^\ell(m_\ell\, a_\ell)| - |\eta^\ell(m'_\ell)|
        \leq 0.
    \end{align*}
    Then $n\leq n'$. If $n=n'$, we obtain a contradiction from part (i).
    Thus, we conclude that $n<n'$.
    
    Now suppose that $n<n'$
    and suppose by contradiction that
        $\tail_{\eta,p,x}(n)
        \not<_{lex} \tail_{\eta,p,x}(n')$.
    If $\tail_{\eta,p,x}(n)
        =  \tail_{\eta,p,x}(n')$,
        then we obtain from part (i) that $n=n'$, a contradiction.
    If $\tail_{\eta,p,x}(n)
        >_{lex} \tail_{\eta,p,x}(n')$,
        then we obtain from above that $n>n'$, a contradiction.
    Therefore, we conclude that
        $\tail_{\eta,p,x}(n)
        <_{lex} \tail_{\eta,p,x}(n')$.
\end{proof}

\section{Dumont--Thomas complement numeration systems for $\Z$ based on periodic points}
\label{sec:dumont-thomas-Z}

In this section, we prove extensions of
Theorem~\ref{thm:dumont-thomas} to right-infinite and left-infinite periodic
points of substitutions from which we deduce a numeration system for $\Z$
associated to any two-sided periodic point with growing seed of a substitution.

\begin{theorem}\label{thm:dumont-thomas-right-p}
    Let $\eta:A^*\to A^*$ be a substitution with growing letter $a\in A$.
    Let $u\in \Per_{\Zp}(\eta)$ such that $u_0 = a$.
    Let $p\geq 1$ be a period of $u$.
    For every integer $n\geq1$,
    there exists a unique integer $k=k(n)$
    such that $p$ divides $k$
    and a unique sequence $(m_i,a_i)_{i=0,\dots,k-1}$ such that
    \begin{enumerate}[(i)]
        \item this sequence is $a$-admissible and 
            $m_{k-1} m_{k-2} \cdots m_{k-p}\neq\varepsilon$,
        \item $u_0u_1\cdots u_{n-1} 
            = \eta^{k-1}(m_{k-1}) \eta^{k-2}(m_{k-2}) \cdots \eta^0(m_0)$.
    \end{enumerate}
\end{theorem}

\begin{proof}
    Since $u$ is a periodic point of period $p$, we have 
    that $u_0=a$ is a prefix of $\eta^p(a)$.
    Also, since $a$ is growing, we have that $\eta^p(a)\in a A^+$.
    Thus $(|\eta^{p\ell}(a)|)_{\ell\in\N}$ is a strictly increasing sequence
    starting with value 1 when $\ell=0$.
    Let $n\geq 1$ be an integer. 
    There exists a unique integer 
    $\ell\geq1$ such that
    $ |\eta^{p(\ell-1)}(a)| \leq n < |\eta^{p\ell}(a)|$.
    Let $k=p\ell$ so that we have
    \begin{equation}\label{eq:interval-for-n-in-Theorem11}
        |\eta^{k-p}(a)| \leq n < |\eta^{k}(a)|.
    \end{equation}
    The word $m = u_0u_1\cdots u_{n-1}$ is thus a proper prefix
    of $\eta^{k}(a)$.
    From Lemma~\ref{lem:when-m-prefix-of-eta-p-a},
    there exists a~unique $a$-admissible sequence
    $(m_i,a_i)_{i=0,\dots,k-1}$ such that
    \[
        m = \eta^{k-1}(m_{k-1}) \eta^{k-2}(m_{k-2}) \cdots \eta^0(m_0).
    \]
    Assume by contradiction that $m_{k-1} m_{k-2} \cdots m_{k-p}=\varepsilon$. 
    Then $a_{k-p}=a$ and
    from Lemma~\ref{lem:admissible-sequence-sum}, we have
    \begin{align*}
        n &= |m| 
        = \textstyle\sum_{j=0}^{k-1}|\eta^j(m_j)|
        = \textstyle\sum_{j=0}^{k-p-1}|\eta^j(m_j)|\\
        &< |\eta^{k-p-1}(m_{k-p-1}a_{k-p-1})|
        \leq |\eta^{k-p-1}(\eta(a_{k-p}))|
        = |\eta^{k-p}(a)|,
    \end{align*}
    a contradiction with~\eqref{eq:interval-for-n-in-Theorem11}.
    Thus $m_{k-1} m_{k-2} \cdots m_{k-p}\neq\varepsilon$.
\end{proof}


We now adapt Dumont--Thomas's theorem to the left-infinite periodic points.

\begin{theorem}\label{thm:dumont-thomas-for-left-side-p}
    Let $\eta:A^*\to A^*$ be a substitution with growing letter $b\in A$.
    Let $u\in \Per_{\Zn}(\eta)$ such that $u_{-1} = b$.
    Let $p\geq 1$ be a period of $u$.
    For every integer $n\leq -2$,
    there exists a unique integer $k=k(n)$ 
    such that $p$ divides $k$
    and a unique sequence
    $(m_i,a_i)_{i=0,\dots,k-1}$ such that
    \begin{enumerate}[(i)]
        \item this sequence is $b$-admissible and 
            \begin{equation}\label{eq:left-periodic-condition}
                \eta^{p-1}(m_{k-1})\eta^{p-2}(m_{k-2})\cdots\eta^0(m_{k-p})a_{k-p}
                \neq\eta^p(b),
            \end{equation}
        \item $u_{-|\eta^{k}(b)|}\cdots u_{n-2}u_{n-1} 
            = \eta^{k-1}(m_{k-1}) \eta^{k-2}(m_{k-2}) \cdots \eta^0(m_0)$.
    \end{enumerate}
\end{theorem}

\begin{proof}
    Since $u$ is a periodic point of period $p$, we have 
    that $u_{-1}=b$ is a suffix of $\eta^p(b)$.
    Also, since $b$ is growing, we have that $\eta^p(b)\in A^+b$.
    Thus $(-|\eta^{p\ell}(b)|)_{\ell\in\N}$ is a strictly decreasing sequence
    starting with value $-1$ when $\ell=0$.
    Let $n\leq -2$ be an integer. 
    There exists a unique integer 
    $\ell\geq1$ such that
    $-|\eta^{p\ell}(b)| \leq n < -|\eta^{p(\ell-1)}(b)| $.
    Let $k=p\ell$ so that we have
    \begin{equation}\label{eq:interval-for-n-in-Theorem12}
        -|\eta^{k}(b)|\leq n  < -|\eta^{k-p}(b)|.
    \end{equation}
    Therefore the word 
    $m = u_{-|\eta^{k}(b)|}\cdots u_{n-2}u_{n-1}$ of length
    \begin{equation}\label{eq:upper-bound-for-m-in-Thm-12}
    |m| = |\eta^{k}(b)| + n 
        < |\eta^{k}(b)| - |\eta^{k-p}(b)| 
        \leq |\eta^{k}(b)|
    \end{equation}
    is a proper prefix of the word $\eta^{k}(b)$.
    From Lemma~\ref{lem:when-m-prefix-of-eta-p-a},
    there exists a unique $b$-admissible sequence
    $(m_i,a_i)_{i=0,\dots,k-1}$ such that
    \[
        m = \eta^{k-1}(m_{k-1}) \eta^{k-2}(m_{k-2}) \cdots \eta^0(m_0).
    \]
    By contradiction, assume that \eqref{eq:left-periodic-condition}
    is an equality. Then $a_{k-p}=b$ and
    \begin{equation*}
        \begin{aligned}
        |m| 
        &= |\eta^{k-p}(\eta^p(b))| - |\eta^{k-p}(a_{k-p})|
          + \textstyle\sum_{j=0}^{k-p-1}|\eta^{j}(m_{j})|\\
        &\geq |\eta^{k}(b)| - |\eta^{k-p}(b)|,
        \end{aligned}
    \end{equation*}
    a contradiction with~\eqref{eq:upper-bound-for-m-in-Thm-12}.
\end{proof}

We may now define a numeration system for $\Z$
using the previous results.

\begin{definition}[Dumont--Thomas complement numeration systems for $\Z$]\label{def:rep-not-prolongable}
    Let $\eta:A^*\to A^*$ be a~substitution and
    $u\in\Per(\eta)$ be a two-sided periodic point with growing seed $u_{-1}|u_0$.
    Let $p\geq1$ be the period of $u$.
    Let $\Dcal=\left\{\0,...,\max_{c\in A}|\eta(c)|-1\right\}$.
    We define
	\[
	\begin{array}{rccl}
        \rep_{u}:&\Z & \to & \{\0,\1\}\odot\Dcal^*\\
		&n & \mapsto &
		\begin{cases}
            \0\odot|m_{k-1}|\odot|m_{k-2}|\odot\ldots\odot|m_0|,& \text{ if } n \geq 1;\\
			\0,                                         & \text{ if } n = 0;\\
			\1,                                         & \text{ if } n =-1;\\
			\1\odot|m_{k-1}|\odot|m_{k-2}|\odot\ldots\odot|m_0|, & \text{ if } n \leq -2,
		\end{cases}
	\end{array}
	\]
    where $k=k(n)\geq0$ is the unique integer and
    $(m_i,a_i)_{i=0,\dots,k-1}$ is the unique sequence
    obtained from Theorem~\ref{thm:dumont-thomas-right-p}
    (Theorem~\ref{thm:dumont-thomas-for-left-side-p}) applied 
    on the right-infinite periodic point $u|_\Zp$
    (on the left-infinite periodic point $u|_\Zn$)
    if $n\geq 1$
    (if $n\leq -2$, respectively)
    both with period $p$.
\end{definition}

\noindent
Note that the period $p\in\N$ of $u$
divides $|\rep_{u}(n)|{-}1$ for every $n\in\Z$.
Also, one may observe that 
\begin{align}
    \rep_u(n) =
		\begin{cases}
            \0\odot\tail_{\eta,k,u_0}(n),    &\text{ if } n\geq0;\\
            \1\odot\tail_{\eta,k,u_{-1}}(n), &\text{ if } n<0.
		\end{cases}
\end{align}

\begin{remark}
    In Definition~\ref{def:rep-not-prolongable},
    the numeration system $\rep_u$ could be defined with any
    period of the two-sided periodic point $u$
    and the main result, Theorem~\ref{th:periodic_point_automatic},
    would still hold.
    A choice is made here to keep it simple and
    always take the period of the periodic point $u$.
\end{remark}

\begin{remark}
    If $u\in\Per(\eta)$ is a two-sided periodic point of period $p$ with growing seed,
    then its restriction $u|_\Zp$
    to the nonnegative integers is also a periodic point, but its period might
    be smaller than $p$ (in general, a divisor of $p$).
    For example, this is what happens for the 
    Fibonacci substitution $\varphi: a\mapsto ab, b\mapsto a$
    or the Thue-Morse substitution $\psi_{TM}: a\mapsto ab, b\mapsto ba$.
    Both have two-sided periodic points of period $2$
    and right-infinite fixed points.
    In Definition~\ref{def:rep-not-prolongable},
    the numeration system is defined with the period of the two-sided periodic
    point $u$ when applying Theorem~\ref{thm:dumont-thomas-right-p}
    on $u|_\Zp$
    and Theorem~\ref{thm:dumont-thomas-for-left-side-p}
    on $u|_\Zn$.
\end{remark}

When $u=\eta^p(u)$ is a periodic point of a substitution $\eta$,
then it is also a fixed point of the substitution $\eta^p$.
Thus, Theorem~\ref{thm:dumont-thomas} may be used to define a
numeration system for $\N$, but it
leads to a much larger alphabet size $\#\Dcal$.
One advantage of Definition~\ref{def:rep-not-prolongable} is
that the size of the alphabet $\Dcal$ is independent of the period $p$.

\begin{example}
    Consider the Tribonacci substitution $\psi_{T}: a\mapsto ab, b\mapsto ac,
    c\mapsto a$ \cite{Ra82}.
    The successive images of the seed $c|a$ under the substitution $\psi_T$ are illustrated below
    in a tree.
\begin{center}
\begin{tikzpicture}[yscale=1.5,>=latex]
    \node[above,rectangle,draw] (start) at (-.5,1) {\texttt{start}};
    \node[above,rectangle,draw] (00a) at (0,0)  {$a$};
    \node[above,rectangle,draw] (10a) at (0,-1) {$a$};
    \node[above,rectangle,draw] (11b) at (1,-1) {$b$};
    \node[above,rectangle,draw] (20a) at (0,-2) {$a$};
    \node[above,rectangle,draw] (21b) at (1,-2) {$b$};
    \node[above,rectangle,draw] (22a) at (2,-2) {$a$};
    \node[above,rectangle,draw] (23c) at (3,-2) {$c$};
    \node[above,rectangle,draw] (30a) at (0,-3) {$a$};
    \node[above,rectangle,draw] (31b) at (1,-3) {$b$};
    \node[above,rectangle,draw] (32a) at (2,-3) {$a$};
    \node[above,rectangle,draw] (33c) at (3,-3) {$c$};
    \node[above,rectangle,draw] (34a) at (4,-3) {$a$};
    \node[above,rectangle,draw] (35b) at (5,-3) {$b$};
    \node[above,rectangle,draw] (36a) at (6,-3) {$a$};
    \node[above,rectangle,draw] (-00c) at (-1,0)  {$c$};
    \node[above,rectangle,draw] (-11a) at (-1,-1) {$a$};
    \node[above,rectangle,draw] (-20a) at (-2,-2) {$a$};
    \node[above,rectangle,draw] (-21b) at (-1,-2) {$b$};
    \node[above,rectangle,draw] (-30a) at (-4,-3) {$a$};
    \node[above,rectangle,draw] (-31b) at (-3,-3) {$b$};
    \node[above,rectangle,draw] (-32a) at (-2,-3) {$a$};
    \node[above,rectangle,draw] (-33c) at (-1,-3) {$c$};
    \node[above] at (-5,-3.5) {$\cdots$};
    \node[above] at (-4,-3.5) {$-4$};
    \node[above] at (-3,-3.5) {$-3$};
    \node[above] at (-2,-3.5) {$-2$};
    \node[above] at (-1,-3.5) {$-1$};
    \node[above] at (0,-3.5) {$0$};
    \node[above] at (1,-3.5) {$1$};
    \node[above] at (2,-3.5) {$2$};
    \node[above] at (3,-3.5) {$3$};
    \node[above] at (4,-3.5) {$4$};
    \node[above] at (5,-3.5) {$5$};
    \node[above] at (6,-3.5) {$6$};
    \node[above] at (7,-3.5) {$\cdots$};

    \draw(-5.5,-3.1) -- (7.5,-3.1);

    \draw[->] (start) -- node[fill=white] {$\0$} (00a);
    \draw[->] (start) -- node[fill=white] {$\1$} (-00c);

    \draw[->] (00a) -- node[fill=white,inner sep=2pt] {$\0$} (10a);
    \draw[->] (00a) -- node[fill=white,inner sep=2pt] {$\1$} (11b);
    \draw[->] (10a) -- node[fill=white,inner sep=2pt] {$\0$} (20a);
    \draw[->] (10a) -- node[fill=white,inner sep=2pt] {$\1$} (21b);
    \draw[->] (11b) -- node[fill=white,inner sep=2pt] {$\0$} (22a);
    \draw[->] (11b) -- node[fill=white,inner sep=2pt] {$\1$} (23c);
    \draw[->] (20a) -- node[fill=white,inner sep=2pt] {$\0$} (30a);
    \draw[->] (20a) -- node[fill=white,inner sep=2pt] {$\1$} (31b);
    \draw[->] (21b) -- node[fill=white,inner sep=2pt] {$\0$} (32a);
    \draw[->] (21b) -- node[fill=white,inner sep=2pt] {$\1$} (33c);
    \draw[->] (22a) -- node[fill=white,inner sep=2pt] {$\0$} (34a);
    \draw[->] (22a) -- node[fill=white,inner sep=2pt] {$\1$} (35b);
    \draw[->] (23c) -- node[fill=white,inner sep=2pt] {$\0$} (36a);

    \draw[->] (-00c) -- node[fill=white,inner sep=1pt] {$\0$} (-11a);
    \draw[->] (-11a) -- node[fill=white,inner sep=1pt] {$\0$} (-20a);
    \draw[->] (-11a) -- node[fill=white,inner sep=1pt] {$\1$} (-21b);
    \draw[->] (-20a) -- node[fill=white,inner sep=1pt] {$\0$} (-30a);
    \draw[->] (-20a) -- node[fill=white,inner sep=1pt] {$\1$} (-31b);
    \draw[->] (-21b) -- node[fill=white,inner sep=1pt] {$\0$} (-32a);
    \draw[->] (-21b) -- node[fill=white,inner sep=1pt] {$\1$} (-33c);
\end{tikzpicture}
\end{center}
Let $\omega=\cdots abac|abacaba\cdots$ be the two-sided periodic point of $\psi_T$ of period $3$ with
seed $c|a$.
In the above figure, the representation $\rep_\omega(n)$ of $n$ 
labels the shortest path from the root of the tree to a node at $x$-position $n\in\N$.
The representation of small integers based on the periodic point $\omega$ is
illustrated in the following table.
    \[
\begin{array}{l|r||l|r}
    n & \rep_\omega(n) & n & \rep_\omega(n)\\
\hline
-7 & \1\0\1\0\1\0\0  & 0 & \0\\
-6 & \1\0\1\0\1\0\1  & 1 & \0\0\0\1\\
-5 & \1\0\1\0\1\1\0  & 2 & \0\0\1\0\\
-4 & \1\0\0\0        & 3 & \0\0\1\1\\
-3 & \1\0\0\1        & 4 & \0\1\0\0\\
-2 & \1\0\1\0        & 5 & \0\1\0\1\\
-1 & \1              & 6 & \0\1\1\0\\
\end{array}
\]
\end{example}

\begin{definition}[quotient, remainder]\label{def:quotient-p}
    Let $\eta:A^*\to A^*$ be a substitution and
    $u\in\Per(\eta)$ be a two-sided periodic point with growing seed
    $s = u_{-1}|u_0$.
    Let $p\geq1$ be the period of $u$.
    Let $n\in\Z\setminus\{-1,0\}$ be an integer
    and $k=k(n)$ be the unique integer and
    $(m_i,a_i)_{i=0,\dots,k-1}$ be the unique sequence
    obtained from Theorem~\ref{thm:dumont-thomas-right-p}
    (Theorem~\ref{thm:dumont-thomas-for-left-side-p}) applied on $u|_\Zp$
    ($u|_\Zn$)
    if $n\geq 1$
    (if $n\leq -2$, respectively)
    both with period $p$.
    We define the \emph{$u$-quotient} of $n$ as
    \[
        q=
		\begin{cases}
            |\eta^{k-p-1}(m_{k-1}) \eta^{k-p-2}(m_{k-2}) \cdots \eta^0(m_p)|,
			            & \text{ if } n \geq 1;\\
            |\eta^{k-p-1}(m_{k-1}) \eta^{k-p-2}(m_{k-2}) \cdots \eta^0(m_p)| -|\eta^{k-p}(u_{-1})|,
			            & \text{ if } n \leq -2;
		\end{cases}
    \]
    and the \emph{$u$-remainder} of $n$  as 
    $r=|\eta^{p-1}(m_{p-1}) \eta^{p-2}(m_{p-2})\cdots\eta^0(m_0)|$.
\end{definition}

Notice that the $u$-quotient $q$ and $u$-remainder $r$ of an integer
$n\in\Z\setminus\{-1,0\}$ fulfill the condition that if $n \geq1$ then $0\leq q < n$ and if $n \leq
-2$ then $n<q\leq -1$. Consequently, $|q| < |n|$.  Also, if $\eta$ is $d$-uniform,
then the $u$-quotient and $u$-remainder of $n$ correspond to the quotient and
remainder of the division of $n$ by~$d^p$. 

\begin{remark}
Note that if we know the $u$-quotient $q$ and the $u$-remainder $r$, we
can recover the sequence $|m_{p-1}|\odot|m_{p-2}|\odot\ldots\odot|m_0|$.
Indeed, it is equal to $\tail_{\eta,p,u_q}(r)$.
\end{remark}

\begin{lemma}\label{lem:rep_recurrence-p}
    Let $\eta:A^*\to A^*$ be a substitution and
    $u\in\Per(\eta)$ be a two-sided periodic point with growing seed.
    Let $p\geq1$ be the period of $u$.
    Let $n\in\Z\setminus\{-1,0\}$ be an integer.
    If $q\in\Z$ is the $u$-quotient 
    and $r\in\N$ is the $u$-remainder
    of $n$, then
    \begin{equation*}
        u_n  =  \eta^p(u_q)[r] 
        \quad
        \text{ and }
        \quad
        \rep_{u}(n) = \rep_{u}(q)\odot \tail_{\eta,p,u_q}(r).
    \end{equation*}
\end{lemma}

\begin{proof}
    Let $a,b\in A$ denote the letters $b=u_{-1}=b$ and $a=u_0$.
    Let $n\in\Z\setminus\{-1,0\}$
    and let $q$ be the $u$-quotient and $r$ the $u$-remainder of~$n$.

    Suppose $n\geq 1$.
    From Theorem~\ref{thm:dumont-thomas-right-p},
    there exists a unique $a$-admissible sequence
    $(m_i,a_i)_{i=0,\dots,k-1}$
    such that
    $u_0 \dots u_{n-1} = \eta^{k-1}(m_{k-1}) \dots \eta^0(m_0)$.
    Also, $\eta^{k-1}(m_{k-1}) \dots \eta^0(m_0)a_0$
    is a prefix of $\eta^{k}(a)$, which is a prefix of $u_0 u_1\cdots u_{|\eta^{k}(a)|-1}$,
    thus $u_n=a_0$.
    Since $u$ has period $p$, the word
    \[
        \eta^{k-p-1}(m_{k-1}) \eta^{k-p-2}(m_{k-2}) \cdots \eta^0(m_p) a_p
    \]
    is a prefix of $\eta^{k-p}(a)$, which is a prefix of $u_0 u_1\cdots u_{|\eta^{k-p}(a)|-1}$.
    Thus $a_p=u_q$.
    Since $\eta^{p-1}(m_{p-1})\cdots \eta^{0}(m_{0})a_0$ is a prefix of
    $\eta^p(a_p)$, we deduce that 
    $u_n = a_0 = \eta^p(a_p)[r] = \eta^p(u_q)[r]$.

    Suppose $n\leq -2$.
    From Theorem~\ref{thm:dumont-thomas-for-left-side-p},
    there exists a unique $b$-admissible sequence
    $(m_i,a_i)_{i=0,\dots,k-1}$ such that
    $u_{-|\eta^{k}(b)|} \dots u_{n-1} 
    = \eta^{k-1}(m_{k-1}) \dots \eta^0(m_0)$.
    Also, $\eta^{k-1}(m_{k-1}) \dots \eta^0(m_0)a_0$
    is a prefix of $\eta^{k}(b)$, 
    which is a prefix of
    $u_{-|\eta^{k}(b)|} \dots u_{-1}$, thus $u_n=a_0$.
    Since $u$ has period $p$, the word
    \[
        \eta^{k-p-1}(m_{k-1}) \eta^{k-p-2}(m_{k-2}) \cdots \eta^0(m_p) a_p
    \]
    is a prefix of $\eta^{k-p}(a)$, which is a prefix of
    $u_{-|\eta^{k-p}(b)|} \dots u_{-1}$, thus $a_p=u_q$.
    Since $\eta^{p-1}(m_{p-1})\cdots \eta^{0}(m_{0})a_0$ is a prefix of
    $\eta^p(a_p)$, we deduce that 
    $u_n = a_0 = \eta^p(a_p)[r] = \eta^p(u_q)[r]$.

    To finish the proof for both cases simultaneously, if $n\geq 1$ ($n\leq -2$), 
    applying Theorem~\ref{thm:dumont-thomas-right-p} 
    (Theorem~\ref{thm:dumont-thomas-for-left-side-p})
    on the $u$-quotient $q$ gives for $\mathtt{d} = \0$ ($\mathtt{d} = \1$)
    \[
    \rep_{u}(q) = \mathtt{d}\odot |m_{k-1}|\odot|m_{k-2}|\odot\ldots\odot|m_p|.
    \]
    As $n\geq 1$ if and only if $q\geq 0$, we have
    \begin{equation*}
        \begin{aligned}
            \rep_{u}(n)
            &=\mathtt{d}\odot|m_{k-1}|\odot|m_{k-2}|\odot\ldots\odot|m_p|\odot 
                             |m_{p-1}|\odot \ldots\odot |m_0| \\
            &= \rep_{u}(q)\odot |m_{p-1}|\odot \ldots\odot |m_0| 
            = \rep_{u}(q)\odot \tail_{\eta,p, u_q}(r).\qedhere
        \end{aligned}
    \end{equation*}
\end{proof}

\section{More examples}\label{sec:examples}

We consider the following substitutions:
\[
\begin{array}{ccccc}
\psi_{TM}: 
\left\{\begin{array}{l}
a\mapsto ab,\\
b\mapsto ba;
\end{array}\right.
    &
\psi_{2}: 
\left\{\begin{array}{l}
a\mapsto ab,\\
b\mapsto cb,\\
c\mapsto ac;
\end{array}\right.
    &
\varphi: 
\left\{\begin{array}{l}
a\mapsto ab,\\
b\mapsto a;
\end{array}\right.
    &
\psi_{T}: 
\left\{\begin{array}{l}
a\mapsto ab, \\
b\mapsto ac, \\
c\mapsto a;
\end{array}\right.
    &
\rho: 
\left\{\begin{array}{l}
a\mapsto ac, \\
b\mapsto cb, \\
c\mapsto c.
\end{array}\right.\\[6mm]
      \text{(Thue-Morse)}
    & \text{(some 2-uniform)}
    & \text{(Fibonacci)}
    & \text{(Tribonacci)}
    & \text{(non-primitive)}\\
\end{array}
\]
We let 
\begin{itemize}
    \item $\alpha\in\Per(\psi_{TM})$ denote the periodic point with the seed $a|a$ and period $2$;
    \item $\beta\in\Per(\psi_{2})$ denote the periodic point with the seed $b|a$ and period $1$;
    \item $\gamma,\delta\in\Per(\varphi)$ denote the periodic point of period 2
        with, respectively, the seeds $b|a$ and $a|a$;
    \item $\tau\in\Per(\psi_T)$ denote the periodic point with the seed $c|a$ and period $3$;
    \item $\chi\in\Per(\mu)$ denote the periodic point with the seed $c|a$ and period $1$
        of the substitution $\mu:a\mapsto abc, b\mapsto c, c\mapsto ac$ defined in the introduction;
    \item $\xi\in\Per(\rho)$ denote the periodic point with the seed $b|a$ and period $1$.
\end{itemize}
The~numeration systems derived from these two-sided periodic points
are shown in Table~\ref{tab:num_systems}.

\begin{table}[h]
\begin{center}
    \scriptsize
    \begin{tabular}{c||r|r|r|r|r|r|r}
        substitution & Thue-Morse & 2-uniform & Fibonacci & Fibonacci & Tribonacci & Introduction & non-primitive\\
        images & $(ab,ba)$ & $(ab,cb,ac)$ & $(ab, a)$ & $(ab, a)$ & $(ab,ac,a)$ & $(abc,c,ac)$ & $(ac,cb,c)$ \\
        \hline
        periodic point & $\alpha$ & $\beta$ & $\gamma$ & $\delta$ & $\tau$ & $\chi$ & $\xi$ \\
        seed           & $a|a$    & $b|a$   & $b|a$ & $a|a$ & $c|a$ & $c|a$ & $b|a$ \\
        period         & $2$      & $1$     & $2$  & $2$  & $3$  & $1$ & $1$ \\
        \hline
        $n$ & $\rep_{\alpha}(n)$ & $\rep_{\beta}(n)$ & $\rep_{\gamma}(n)$ & $\rep_{\delta}(n)$ & $\rep_{\tau}(n)$ & $\rep_{\chi}(n)$ & $\rep_{\xi}(n)$\\ 
        \hline
        &&&&&&\\[-2mm]
         10 &  \texttt{01010}   & \texttt{01010}  &  \texttt{0010010}   & \texttt{0010010} & \texttt{0001011} & \texttt{0202} & \texttt{01000000000} \\
         9  &  \texttt{01001}   & \texttt{01001}  &  \texttt{0010001}   & \texttt{0010001} & \texttt{0001010} & \texttt{0201} & \texttt{0100000000} \\
         8  &  \texttt{01000}   & \texttt{01000}  &  \texttt{0010000}   & \texttt{0010000} & \texttt{0001001} & \texttt{0200} & \texttt{010000000} \\
         7  &  \texttt{00111}   & \texttt{0111}   &  \texttt{01010}     & \texttt{01010}   & \texttt{0001000} & \texttt{0101} & \texttt{01000000} \\
         6  &  \texttt{00110}   & \texttt{0110}   &  \texttt{01001}     & \texttt{01001}   & \texttt{0110}    & \texttt{0100} & \texttt{0100000}    \\
         5  &  \texttt{00101}   & \texttt{0101}   &  \texttt{01000}     & \texttt{01000}   & \texttt{0101}    & \texttt{021}  & \texttt{010000}    \\
         4  &  \texttt{00100}   & \texttt{0100}   &  \texttt{00101}     & \texttt{00101}   & \texttt{0100}    & \texttt{020}  & \texttt{01000}    \\
         3  &  \texttt{011}     & \texttt{011}    &  \texttt{00100}     & \texttt{00100}   & \texttt{0011}    & \texttt{010}  & \texttt{0100}    \\
         2  &  \texttt{010}     & \texttt{010}    &  \texttt{010}       & \texttt{010}     & \texttt{0010}    & \texttt{02}   & \texttt{010}    \\
         1  &  \texttt{001}     & \texttt{01}     &  \texttt{001}       & \texttt{001}     & \texttt{0001}    & \texttt{01}   & \texttt{01}    \\
         0  &  \texttt{0}       & \texttt{0}      &  \texttt{0}         & \texttt{0}       & \texttt{0}       & \texttt{0}    & \texttt{0}       \\
        -1  &  \texttt{1}       & \texttt{1}      &  \texttt{1}         & \texttt{1}       & \texttt{1}       & \texttt{1}    & \texttt{1}       \\
        -2  &  \texttt{110}     & \texttt{10}     &  \texttt{100}       & \texttt{101}     & \texttt{1010}    & \texttt{10}   & \texttt{10}    \\
        -3  &  \texttt{101}     & \texttt{101}    &  \texttt{10010}     & \texttt{100}     & \texttt{1001}    & \texttt{102}  & \texttt{100}    \\
        -4  &  \texttt{100}     & \texttt{100}    &  \texttt{10001}     & \texttt{10101}   & \texttt{1000}    & \texttt{101}  & \texttt{1000}    \\
        -5  &  \texttt{11011}   & \texttt{1011}   &  \texttt{10000}     & \texttt{10100}   & \texttt{1010110} & \texttt{100}  & \texttt{10000} \\
        -6  &  \texttt{11010}   & \texttt{1010}   &  \texttt{1001010}   & \texttt{10010}   & \texttt{1010101} & \texttt{1021} & \texttt{100000} \\
        -7  &  \texttt{11001}   & \texttt{1001}   &  \texttt{1001001}   & \texttt{10001}   & \texttt{1010100} & \texttt{1020} & \texttt{1000000} \\
        -8  &  \texttt{11000}   & \texttt{1000}   &  \texttt{1001000}   & \texttt{10000}   & \texttt{1010011} & \texttt{1010} & \texttt{10000000} \\
        -9  &  \texttt{10111}   & \texttt{10111}  &  \texttt{1000101}   & \texttt{1010101} & \texttt{1010010} & \texttt{1002} & \texttt{100000000} \\
        -10 &  \texttt{10110}   & \texttt{10110}  &  \texttt{1000100}   & \texttt{1010100} & \texttt{1010001} & \texttt{1001} & \texttt{1000000000} \\
    \end{tabular}
\end{center}
    \caption{Numeration systems for periodic points
    $\alpha$, $\beta$, $\gamma$, $\delta$, $\tau$, $\chi$, $\xi$ with given seed.}
    \label{tab:num_systems}
\end{table}

\section{Periodic points as Automatic Sequences}\label{sec:periodic-as-automatic}

Let $\eta:A^*\to A^*$ be a substitution and
$u\in\Per(\eta)$ be a two-sided periodic point with growing seed
$s = u_{-1}|u_0$.
Let $\Dcal=\left\{\0,...,\max_{c\in A}|\eta(c)|-1\right\}$.
We associate an automaton $\Acal_{\eta,s}$ with $(\eta,s)$
by adding a new state $\texttt{start}$ and two additional
edges to the automaton $\Acal_{\eta, a}$ defined in~\cite{MR2742574}.
The automaton $\Acal_{\eta,s} = (A\cup\left\{\texttt{start}\right\},\Dcal,
\delta, \texttt{start}, A)$ has the transition function 
$\delta:A\cup\{\texttt{start}\}\to A$ such that 
\begin{itemize}
    \item $\delta(\texttt{start},\0) = s_0 =u_0$, $\quad\delta(\texttt{start},\1) =s_{-1}= u_{-1}$,
    \item for every $c,d\in A$, 
        every $w=w_0 w_1\dots w_{\ell-1}\in A^\ell$ and every $i\in \Dcal$,
        it holds that
        $\delta(c,i)=d$ if and only if $\eta(c)=w$ and $w_i=d$. 
\end{itemize}

Examples of automata associated to the Fibonacci substitution are shown in
Figure~\ref{fig:automataZ}.
\begin{figure}[h]
	\begin{center}
		\begin{tikzpicture}[auto]
		\begin{scope}[xshift=-.1cm]
		\node[draw,circle] (A) at (0,0) {$a$};
		\node[draw,circle] (B) at (1.5,0) {$b$};
		\draw[bend left,->] (A) to node {\scriptsize 1} (B);
		\draw[bend left,->] (B) to node {\scriptsize 0} (A);
		\draw[loop left,->] (A) to node {\scriptsize 0} (A);
		\draw[<-] (A) -- ++ (0,1) node[above] {};
		\draw[->] (A) -- ++ (0,-0.5);
		\draw[->] (B) -- ++ (0,-0.5);
		\end{scope}
		\begin{scope}[xshift=3.5cm]
		\node[draw,circle] (A) at (0,0) {$a$};
		\node[draw,circle] (B) at (1.5,0) {$b$};
		\draw[bend left,->] (A) to node {\scriptsize 1} (B);
		\draw[bend left,->] (B) to node {\scriptsize 0} (A);
		\draw[loop left,->] (A) to node {\scriptsize 0} (A);
		\draw[<-] (B) -- ++ (0,1) node[above] {};
		\draw[->] (A) -- ++ (0,-0.5);
		\draw[->] (B) -- ++ (0,-0.5);
		\end{scope}
		\begin{scope}[xshift=7cm]
		\node[draw] (S) at (.75,1.2) {$\texttt{start}$};
		\node[draw,circle] (A) at (0,0) {$a$};
		\node[draw,circle] (B) at (1.5,0) {$b$};
		\draw[bend left=20,->] (A) to node {\scriptsize 1} (B);
		\draw[bend left=20,->] (B) to node {\scriptsize 0} (A);
		\draw[bend right,->] (S) to node[left] {\scriptsize 0} (A);
		\draw[bend left,->] (S) to node[right] {\scriptsize 1} (B);
		\draw[<-] (S) -- ++ (0,.5) node[above] {};
		\draw[loop left,->] (A) to node {\scriptsize 0} (A);
		\draw[->] (A) -- ++ (0,-0.5);
		\draw[->] (B) -- ++ (0,-0.5);
		\end{scope}
		\end{tikzpicture}
	\end{center}
	\caption{Automata $\Acal_{\varphi,a}$, $\Acal_{\varphi,b}$ and 
		$\Acal_{\varphi,s}$ for $\varphi:a\mapsto ab$, $b\mapsto a$ and $s=b|a$.}
	\label{fig:automataZ}
\end{figure}

If the seed is $s=b|a$, the automaton $\Acal_{\eta,s}$
is related to the usual automata
$\Acal_{\eta,a}$ and $\Acal_{\eta,b}$ according to the following 
equalities
for every $w\in\Dcal^*$:
\begin{equation}
\label{eq:automata-relations}
    \Acal_{\eta,s}(\0\odot w) = \Acal_{\eta,a}(w)
    \qquad
    \text{and}
    \qquad
    \Acal_{\eta,s}(\1\odot w) = \Acal_{\eta,b}(w).
\end{equation}
Also if $\Acal_{\eta,s}(w) = a$ for some $w\in\Dcal^+$, then
for every $u\in\Dcal^*$
\begin{equation}
\label{eq:automata-relations-2}
    \Acal_{\eta,a}(u) = \Acal_{\eta,s}(w\odot u).
\end{equation}

A theorem of Cobham \cite{MR457011} says that 
a sequence $u=(u_n)_{n\geq0}$ is $k$-automatic
with $k\geq2$
if and only if it is the image,
under a coding, of a fixed point of a $k$-uniform morphism 
\cite[\S 6]{MR1997038}.
It was extended to 
abstract numeration systems based on regular languages 
which includes numeration systems based on
non-uniform morphisms \cite{zbMATH01916667};
see also 
\cite[\S 3]{MR2742574}.
The following result extends Cobham's theorem to the case of two-sided periodic
points of non-uniform substitutions.

\begin{theorem}\label{th:periodic_point_automatic}
    Let $\eta:A^*\to A^*$ be a substitution and
    $u\in\Per(\eta)$ be a two-sided periodic point with growing seed
    $s = u_{-1}|u_0$.
	Then for every $n\in\Z$
	\[
		u_n = \Acal_{\eta,s}(\rep_{u}(n)).
	\]
\end{theorem}
\begin{proof}
    If $n \in\{0,-1\}$ then by definition 
    we have
    $u_n = s_n = \Acal_{\eta,s}(\rep_{u}(n))$.

    Let $n \in\Z\setminus \{0,-1\}$. 
    Induction hypothesis: for every $m\in\Z$ such that $|m| < |n|$
	it holds that $x_m = \Acal_{\eta,s}(\rep_{u}(m))$.
	Let $q$ be the $u$-quotient and $r$ the $u$-remainder
	of $n$. As $|q| < |n|$, $q$ fulfills the induction hypothesis, i.e.,
    $u_q=\Acal_{\eta,s}(\rep_{u}(q))$.
    Let $p\geq1$ be the period of $u$.
    From Lemma~\ref{lem:rep_recurrence-p} we have
	$u_n  = \eta^p(u_q)[r]$ and 
    $\rep_{u}(n) = \rep_{u}(q)\odot\tail_{\eta,p,u_q}(r)$. 
    Using Lemma~\ref{lem:automaton_q_r} and 
    Equation~\eqref{eq:automata-relations-2}, we have
	\begin{equation*}
		\begin{aligned}
		u_n &= \eta^p(u_q)[r]
             = \Acal_{\eta,u_q}(\tail_{\eta,p,u_q}(r)) \\
            &= \Acal_{\eta,s}(\rep_{u}(q)\odot \tail_{\eta,p,u_q}(r))
			 = \Acal_{\eta,s}(\rep_{u}(n)).\qedhere
		\end{aligned}
	\end{equation*}
\end{proof}

\section{Numeration systems for $\Z^d$ based on periodic points}\label{sec:NS-for-Zd}

A numeration system for $\Z^d$ can be deduced from the numeration system for
$\Z$ based on a~periodic point. Since not all integers are represented by words
of the same length, we propose here a way to pad them to a common length.

Let $\eta:A^*\to A^*$ be a substitution and $u\in \Per(\eta)$ with period $p\geq 1$
and growing seed.
Let $\Wmin$ and $\Wmax$ be the following minimum and the maximum element
under the $\tail$ map with particular parameters:
\begin{equation*}
\begin{aligned}
    \Wmin &= \tail_{\eta,p,u_{0}}(0) = \0^p,\\
    \Wmax &= \tail_{\eta,p,u_{-1}}(|\eta^p(u_{-1})|-1).
\end{aligned}
\end{equation*}
The words $\Wmin$ and $\Wmax$ play the role of neutral words in the
numeration system as illustrated in the next lemma.
Below the words $\Wmin$ and $\Wmax$ are concatenated with others words
from $\Dcal^*$ using the binary operation $\odot$,
which is not explicitly written to avoid heavy notation.

\begin{lemma}
    Let $\eta:A^*\to A^*$ be a substitution and
    $u\in\Per(\eta)$ be a two-sided periodic point with growing seed
    $s = u_{-1}|u_0$.
    Let $w\in\Lcal(\Acal_{\eta,s})$. Then
    \[
		\Acal_{\eta,s}(w) =
        \begin{cases}
            \Acal_{\eta,s}(\0 (\Wmin)^i v), & \text{ if } w=\0 v,\\
            \Acal_{\eta,s}(\1 (\Wmax)^i v), & \text{ if } w=\1 v,
        \end{cases}
    \]
    for every integer $i\geq0$.
\end{lemma}

\begin{proof}
    Let $i\geq0$ be an integer
    and $p\geq1$ be the period of $u$.
    Let $w\in\Lcal(\Acal_{\eta,s})$. 

    Suppose that $w$ starts with letter $\0$.
    Let $v\in\Dcal^*$ such that $w=\0v$.
    We have $\Acal_{\eta,u_0}(\0^p) = u_0$.
    Thus $\Acal_{\eta,s}(\0(\Wmin)^i) = u_0$.
    From Equation~\eqref{eq:automata-relations}
    and Equation~\eqref{eq:automata-relations-2} we obtain
    \[
		\Acal_{\eta,s}(w)
		= \Acal_{\eta,s}(\0 v)
        \stackrel{\eqref{eq:automata-relations}}{=} \Acal_{\eta,u_0}(v)
        \stackrel{\eqref{eq:automata-relations-2}}{=} \Acal_{\eta,s}(\0(\Wmin)^iv).
    \]

    Suppose that $w$ starts with letter $\1$.
    Let $v\in\Dcal^*$ such that $w=\1v$.
    We have $\Acal_{\eta,u_{-1}}(\Wmax) = u_{-1}$.
    Thus $\Acal_{\eta,s}(\1(\Wmax)^i) = u_{-1}$.
    From Equation~\eqref{eq:automata-relations}
    and Equation~\eqref{eq:automata-relations-2} we obtain
    \[
		\Acal_{\eta,s}(w)
		= \Acal_{\eta,s}(\1 v)
        \stackrel{\eqref{eq:automata-relations}}{=} \Acal_{\eta,u_{-1}}(v)
        \stackrel{\eqref{eq:automata-relations-2}}{=} \Acal_{\eta,s}(\1(\Wmax)^iv).\qedhere
    \]
\end{proof}

It is useful to pad words to a certain length using neutral words as follows
using a pad function.
Let $s = u_{-1}|u_0$.
Let $w\in\Lcal_{\ell p+1}(\Acal_{\eta,s})$ for some $\ell\in\N$.
Let $t\in\N$ such that $t\geq |w|$ and $t \bmod p = 1$.
    We define
    \[
        \pad_t(w)
        =
        \begin{cases}
            \0 (\Wmin)^m v,  & \text{ if } w=\0 v;\\
            \1 (\Wmax)^m v,  & \text{ if } w=\1 v,
        \end{cases}
    \]
where $m=(t-|w|)/p$.
The padding map can be used to pad words so that they all
have the same length. This allows us to represent coordinates in
$\Z^d$ in dimension $d\geq1$.

\begin{definition}[Numeration system for $\Z^d$]\label{def:num-sys-Zd}
    Let $\eta:A^*\to A^*$ be a substitution and
    $u_1,u_2,\dots,u_d\in\Per(\eta)$ be periodic points with growing seeds and of the same period.
	For every $\bn=(n_1,n_2,\dots,n_d)\in\Z^d$,
	we define
	\[
	\rep_u(\bn)=\left(
	\begin{array}{c}
        \pad_t(\rep_{u_1}(n_1))\\
        \pad_t(\rep_{u_2}(n_2))\\
        \dots\\
        \pad_t(\rep_{u_d}(n_d))
	\end{array}
	\right)
    \in\{\0,\1\}^d(\Dcal^d)^*,
	\]
    where $t=\max\{|\rep_{u_i}(n_i)|\colon 1\leq i\leq d\}$.
\end{definition}

\begin{remark}
    In Definition~\ref{def:num-sys-Zd},
    considering different periodic points with the same period of the same
    1-dimensional substitution in each dimension
    can be necessary for
    instance to describe the different 2-dimensional periodic points of
    2-dimensional substitutions. This is what happens when one wants to
    describe the 8 configurations of Wang tiles presented in \cite{MR3978536} 
    which are the periodic points of a 2-dimensional substitution.
\end{remark}

Of course, it is possible and simpler to use the same periodic point to represent the
entries of an integer vector.
This is what is done in the example that follows.

\begin{example} 
    Consider the Tribonacci substitution $\psi_{T}: a\mapsto ab, b\mapsto ac,
    c\mapsto a$ as in Example~\ref{ex:Tribo}.
    Let $\tau\in\Per(\psi_{T})$ be the periodic point 
    with period $p=3$ and seed $c|a$.
    We have
\begin{align*}
    \Wmin 
    &= \tail_{\psi_T,p,u_{0}}(0) 
    = \tail_{\psi_T,3,a}(0) 
    = \0^3
    = \0\0\0,\\
    \Wmax 
    &= \tail_{\psi_T,p,u_{-1}}(|\psi_T^p(u_{-1})|-1)
    = \tail_{\psi_T,3,c}(|\psi_T^3(c)|-1)
    = \tail_{\psi_T,3,c}(3)
    = \0\1\1.
\end{align*}
The words $\Wmin$ and $\Wmax$ can be used to pad words to a given length which
    is a multiple of 3 plus 1.
    For instance, we illustrate in the following table the padding of
    the Dumont--Thomas representation based on the periodic point $\tau$.
    The representation of integers from $-10$ to $10$ is padded to words of length 7.
    {\scriptsize
    \begin{center}
    \begin{tabular}{c||r|c}
        \hline
        $n$ & $\rep_{\tau}(n)$   & $\pad_7(\rep_{\tau}(n))$ \\
        \hline
        &&\\[-2mm]
         10 &   \texttt{0001011} &   \texttt{0001011} \\
         9  &   \texttt{0001010} &   \texttt{0001010} \\
         8  &   \texttt{0001001} &   \texttt{0001001} \\
         7  &   \texttt{0001000} &   \texttt{0001000} \\
         6  &   \texttt{0110}    &   \texttt{0000110} \\
         5  &   \texttt{0101}    &   \texttt{0000101} \\
         4  &   \texttt{0100}    &   \texttt{0000100} \\
         3  &   \texttt{0011}    &   \texttt{0000011} \\
         2  &   \texttt{0010}    &   \texttt{0000010} \\
         1  &   \texttt{0001}    &   \texttt{0000001} \\
         0  &   \texttt{0}       &   \texttt{0000000} \\
        -1  &   \texttt{1}       &   \texttt{1011011} \\
        -2  &   \texttt{1010}    &   \texttt{1011010} \\
        -3  &   \texttt{1001}    &   \texttt{1011001} \\
        -4  &   \texttt{1000}    &   \texttt{1011000} \\
        -5  &   \texttt{1010110} &   \texttt{1010110} \\
        -6  &   \texttt{1010101} &   \texttt{1010101} \\
        -7  &   \texttt{1010100} &   \texttt{1010100} \\
        -8  &   \texttt{1010011} &   \texttt{1010011} \\
        -9  &   \texttt{1010010} &   \texttt{1010010} \\
        -10 &   \texttt{1010001} &   \texttt{1010001} \\
    \end{tabular}
    \end{center}
    }
    Thus, the coordinate $(-1,8)\in\Z^2$ can thus be written as a word
	\[
	\rep_\tau(-1,8)
    =\left(\begin{array}{c}
        \pad_7(\rep_\tau(-1))\\
        \pad_7(\rep_\tau(8))\\
	\end{array}\right)
    =\left(\begin{array}{c}
        \1\0\1\1\0\1\1\\
        \0\0\0\1\0\0\1
	\end{array}\right)
    \in
    \left\{
    \left(\begin{array}{c} \0 \\ \0 \end{array}\right),
    \left(\begin{array}{c} \0 \\ \1 \end{array}\right),
    \left(\begin{array}{c} \1 \\ \0 \end{array}\right),
    \left(\begin{array}{c} \1 \\ \1 \end{array}\right)
    \right\}^*
	\]
    whose alphabet of size 4 is the Cartesian product of the alphabet
    $\{\0,\1\}$ with itself.
\end{example}

\section{A Total Order}\label{sec:total-order}

In this section, we define a total order on $\{\0,\1\}\Dcal^*:=\{\0,\1\}\odot\Dcal^*$
and we show that $\rep_{u}$ is increasing with respect to this order.

The radix order on a language $L\subset \Dcal^*$ is a total order $(L,<_{rad})$ 
such that $u <_{rad} v$ if and only if
    $|u| < |v|$ or $|u| = |v|$ and $u <_{lex} v$,
    where $<_{lex}$ denotes the lexicographic order.
For example, over the alphabet $\{\0,\1\}$, the minimum elements for the radix order are:
\[
    \varepsilon
    <_{rad} \0 
    <_{rad} \1 
    <_{rad} \0\0 
    <_{rad} \0\1 
    <_{rad} \1\0
    <_{rad} \1\1
    <_{rad} \0\0\0
    <_{rad} \0\0\1
    <_{rad} \cdots.
\]
We define the reversed-radix order as a total order such that $u <_{rev} v$ if and only if
    $|u| > |v|$ or $|u| = |v|$ and $u <_{lex} v$.
For example, over the alphabet $\{\0,\1\}$, 
the maximum elements for the reverse-radix order are:
\[
    \cdots
    <_{rev} \1\1\0
    <_{rev} \1\1\1
    <_{rev} \0\0 
    <_{rev} \0\1 
    <_{rev} \1\0
    <_{rev} \1\1
    <_{rev} \0
    <_{rev} \1
    <_{rad} \varepsilon.
\]
Let us stress that the reversed-radix order behaves in the same manner
as the radix order on the words of the same length.
Also if $L$ has infinite cardinality, then $L$ has no maximal element for the radix order
and has no minimum element for the reverse-radix order.

The radix order and the reverse-radix order can be used jointly to define a
total order on a~language with no minimum nor maximum element. Below, we use the first letter
of a word in $\{\0,\1\}\Dcal^*$ to split the two cases.

\begin{definition}[total order $\prec$]\label{def:order}
    For every $u,v \in \{\0,\1\}\Dcal^*$,
	we define $u \prec v$ if and only if
	\begin{itemize}
		\item $u\in \1 \Dcal^*$ and $v\in \0 \Dcal^*$, or
        \item $u,v \in \0 \Dcal^*$ and $u <_{rad} v$, or
        \item $u,v \in \1 \Dcal^*$ and $u <_{rev} v$.
	\end{itemize}
\end{definition}
Thus, if $\Dcal=\{\0,\1\}$, we get
\[
    \cdots
    \prec \1\0\0
    \prec \1\0\1
    \prec \1\1\0
    \prec \1\1\1
    \prec \1\0
    \prec \1\1
    \prec \1
    \prec \0 
    \prec \0\0 
    \prec \0\1 
    \prec \0\0\0
    \prec \0\0\1
    \prec \0\1\0
    \prec \0\1\1
    \prec \cdots.
\]
The total order $\prec$ makes sense with respect to Dumont--Thomas complement numeration
systems for $\Z$ because of the following result.

\begin{proposition}\label{prop:rep-u-is-increasing}
    Let $\eta:A^*\to A^*$ be a substitution and
    $u\in\Per(\eta)$ be a two-sided periodic point with growing seed.
    The map
    $\rep_{u}:\Z\to\{\0,\1\}\Dcal^*$
    is increasing with respect to the order $\prec$ on $\{\0,\1\}\Dcal^*$.
\end{proposition}

\begin{proof}
    Let $n,n'\in\Z$ be two integers such that $n<n'$.

    Assume that $n<0\leq n'$. 
    Then $\rep_u(n)\in\1\Dcal^*$
    and $\rep_u(n')\in\0\Dcal^*$
    so that
    $\rep_u(n)\prec\rep_u(n')$.

    Assume that $0\leq n<n'$. 
    Then $\rep_u(n)\in\0\Dcal^*$
    and $\rep_u(n')\in\0\Dcal^*$.
    The case $|\rep_u(n)|>|\rep_u(n')|$ is impossible.
    Indeed, suppose that $|\rep_u(n)|=k+1$ and $|\rep_u(n')|=k'+1$
    for some integers $k$ and $k'$.
    If $|\rep_u(n)|>|\rep_u(n')|$, then $k-p\geq k'$,
    where $p$ is the period of $u$.
    From Equation~\eqref{eq:interval-for-n-in-Theorem11}, we have
    \[
        n' < |\eta^{k'}(a)|
        \leq |\eta^{k-p}(a)|
        \leq n,
    \]
    a contradiction.
    If $|\rep_u(n)|<|\rep_u(n')|$,
    then $\rep_u(n)\prec\rep_u(n')$.
    Suppose now that $|\rep_u(n)|=|\rep_u(n')|=k+1$
    for some integer $k$.
    From Lemma~\ref{lem:lexico-order-on-admissible-seq},
    we have 
    \[
    \rep_u(n) =\0\odot\tail_{\eta,k,u_0}(n)
        <_{lex}\0\odot\tail_{\eta,k,u_0}(n')
        =\rep_u(n').
    \]
    Thus $\rep_u(n)\prec\rep_u(n')$.

    Assume that $n<n'<0$. 
    Then $\rep_u(n)\in\1\Dcal^*$
    and $\rep_u(n')\in\1\Dcal^*$.
    The case $|\rep_u(n)|<|\rep_u(n')|$ is impossible.
    Indeed, suppose that $|\rep_u(n)|=k+1$ and $|\rep_u(n')|=k'+1$
    for some integers $k$ and $k'$.
    If $|\rep_u(n)|<|\rep_u(n')|$, then $k'-p\geq k$,
    where $p$ is the period of $u$.
    From Equation~\eqref{eq:interval-for-n-in-Theorem12}, we have
    \[
        n' 
        < -|\eta^{k'-p}(b)|
        \leq -|\eta^{k}(b)|
        \leq n,
    \]
    a contradiction.
    If $|\rep_u(n)|>|\rep_u(n')|$,
    then $\rep_u(n)\prec\rep_u(n')$.
    Suppose that $|\rep_u(n)|=|\rep_u(n')|=k+1$
    for some integer $k$.
    From Lemma~\ref{lem:lexico-order-on-admissible-seq},
    we have 
    \[
    \rep_u(n) =\1\odot\tail_{\eta,k,u_{-1}}(n)
        <_{lex}\1\odot\tail_{\eta,k,u_{-1}}(n')
        =\rep_u(n').
    \]
    Thus $\rep_u(n)\prec\rep_u(n')$.
\end{proof}

It follows from Proposition~\ref{prop:rep-u-is-increasing}
that $\rep_{u}:\Z\to\{\0,\1\}\Dcal^*$ is injective.
Therefore it is a bijection onto its image.
The next result describes the image of the map $\rep_{u}$.

\begin{lemma}\label{lem:repu-image}
    Let $\eta:A^*\to A^*$ be a substitution and
    $u\in\Per(\eta)$ be a two-sided periodic point with growing seed $s = u_{-1}|u_0$.
    Let $p\geq1$ be the period of $u$.
    Then
    \[
        \rep_u(\Z) = \bigcup_{\ell\in\N}
        \Lcal_{\ell p+1}(\Acal_{\eta,s}) \setminus \{\0\Wmin, \1\Wmax\} \Dcal^*.
    \]
\end{lemma}

\begin{proof}
    $(\subseteq)$.
    It follows from Theorem~\ref{th:periodic_point_automatic} that
    $\rep_u(\Z)\subset\Lcal(\Acal_{\eta,s})$.
    Also for every $n\in\Z$, $\rep_u(n)$ is a word of length $\ell p+1$
    for some $\ell\in\N$.
    Thus $\rep_u(\Z) \subset \bigcup_{\ell\in\N} \Lcal_{\ell p+1}(\Acal_{\eta,s})$.
    It remains to show that
    $\rep_u(\Z)\cap \{\0\Wmin, \1\Wmax\} \Dcal^*=\varnothing$.
    Suppose by contradiction that there exists $n\in\Z$ such that
    $\rep_u(n)\in\0\Wmin\Dcal^*$.
    We have
    $\rep_u(n)=\0\odot |m_{k-1}|\odot |m_{k-2}|\odot\ldots\odot|m_0|$,
    where $k=\ell p$.
    Then 
    $|m_{k-1}|\odot\ldots\odot|m_{k-p}| = \0^p$, which implies
    $m_{k-1} m_{k-2}\ldots m_{k-p} = \varepsilon$, thus contradicting
    Theorem~\ref{thm:dumont-thomas-right-p}.
    On the other hand, suppose by contradiction that there exists $n\in\Z$ such that
    $\rep_u(n)\in\1\Wmax\Dcal^*$.
    We have
    $\rep_u(n)=\1\odot |m_{k-1}|\odot |m_{k-2}|\odot\ldots\odot|m_0|$.
    Then 
    $|m_{k-1}|\odot\ldots\odot|m_{k-p}| =  
    \tail_{\eta,p,u_{-1}}(|\eta^p(u_{-1})|-1)$,
    which implies
    $\eta^{p-1}(m_{k-1})\eta^{p-2}(m_{k-2})\cdots\eta^0(m_{k-p})$
    is the prefix of $\eta^p(u_{-1})$ of length $|\eta^p(u_{-1})|-1$,
    a contradiction with 
    Theorem~\ref{thm:dumont-thomas-for-left-side-p}.

    $(\supseteq)$.
    Let $\ell\in\N$ and $k=\ell p$.
    Let $v= v_{k-1}\cdots v_0$
    such that $\ttd \odot v
    \in\Lcal_{\ell p+1}(\Acal_{\eta,s}) \setminus \{\0\Wmin, \1\Wmax\} \Dcal^*$.

    Suppose that $\ttd=\0$.
    We have $v \in\Lcal(\Acal_{\eta,u_0})$.
    From Lemma~\ref{lem:automaton-language-is-admissible},
    there exists a $u_0$-admissible sequence
    $(m_i,a_i)_{i=0,\dots,k-1}$ such that
    $|m_i|=v_i$ for every $i=0,\dots,k-1$.
    Let $n=\sum_{i=0}^{k-1}|\eta^i(m_i)|$.
    From Theorem~\ref{thm:dumont-thomas-right-p}
    and using $v\notin \Wmin\Dcal^*$, 
    we have $\rep_u(n)=\0\odot v$.
    Thus $\ttd\odot v\in\rep_u(\Z)$.

    Suppose that $\ttd=\1$.
    We have $v\in\Lcal(\Acal_{\eta,u_{-1}})$.
    From Lemma~\ref{lem:automaton-language-is-admissible},
    there exists a $u_{-1}$-admissible sequence
    $(m_i,a_i)_{i=0,\dots,k-1}$ such that
    $|m_i|=v_i$ for every $i=0,\dots,k-1$.
    Let $n=-|\eta^{k}(u_{-1})|+\sum_{i=0}^{k-1}|\eta^i(m_i)|$.
    From Theorem~\ref{thm:dumont-thomas-for-left-side-p},
    and using $v\notin \Wmax\Dcal^*$, 
    we have $\rep_u(n)=\1\odot v$.
    Thus $\ttd\odot v\in\rep_u(\Z)$.
\end{proof}

Results similar to Proposition~\ref{prop:rep-u-is-increasing} exist for other
numeration systems; see \cite[\S 5]{MR1879663} and \cite[\S 4]{MR2290784}.
In some other works on numeration systems, such an increasing bijection 
is not a~consequence but rather a hypothesis.
For example, a bijection $\N\to\Lcal$ serves as the definition of abstract
numeration systems in \cite{MR1799066}.
Similarly, we have the following characterization of Dumont--Thomas complement numeration systems
for $\Z$ in terms of the total order $\prec$ on the language recognized by an automaton.

\begin{theorem}\label{thm:dumont-thomas-characterization-increasing}
    Let $\eta:A^*\to A^*$ be a substitution and
    $u\in\Per(\eta)$ be a two-sided periodic point with growing seed $s = u_{-1}|u_0$.
    Let $p\geq1$ be the period of $u$.
    Let $f:\Z\to\{\0,\1\}\Dcal^*$ be some map.
    The following items are equivalent:
    \begin{itemize}
    \item $f=\rep_u$,
    \item $f$ is increasing with respect to $\prec$, its image is $f(\Z)=
        \bigcup_{\ell\in\N}\Lcal_{\ell p+1}(\Acal_{\eta,s}) \setminus \{\0\Wmin, \1\Wmax\} \Dcal^*$
            and $f(0)=\0$.
    \end{itemize}
\end{theorem}

\begin{proof}
    Suppose that $f=\rep_u$.
    Then $f$ is increasing from Proposition~\ref{prop:rep-u-is-increasing}.
    Its image was computed in Lemma~\ref{lem:repu-image}, 
    Also, $f(0)=\0$ from Definition~\ref{def:rep-not-prolongable}.

    Let $f:\Z\to\{\0,\1\}\Dcal^*$.
    Suppose $f$ is increasing, its image is $f(\Z)=
        \bigcup_{\ell\in\N}\Lcal_{\ell p+1}(\Acal_{\eta,s}) \setminus \{\0\Wmin, \1\Wmax\} \Dcal^*$
            and $f(0)=\0$.
    The map $\rep_u$ satisfies the same properties.
    Since there is a unique increasing bijection $\Z\to f(\Z)$ such that $f(0)=\0$,
    we conclude that $f=\rep_u$.
\end{proof}

\section{Relation with existing complement numeration systems}
\label{sec:known-complement-NS}

In this section, we show that two existing complement numeration systems can be recovered
as a Dumont--Thomas complement numeration system using the some well-chosen substitutions.
The involved substitutions are part of the examples presented in Section~\ref{sec:examples}.

\subsection{Two's complement numeration system}

Let $\Dcal = \{\0,\1\}$.
In the two's complement representation of integers
the value of a binary word $w=w_{k-1} w_{k-2}\cdots w_0\in\Dcal^{k}$
is $\valTwoC(w) = \sum_{i=0}^{k-1}w_i 2^i - w_{k-1}2^{k}$; see \cite[\S 4.1]{MR3077153}.
For every $n\in\Z$ there exists a unique word 
$w\in\Dcal^+\setminus \left( \0\0\Dcal^* \cup \1\1\Dcal^* \right)$
such that $n=\valTwoC(w)$. 
The word $w$ is called the \emph{two's complement representation} of the integer $n$,
and we denote it by $\rep_{2c}(n)$.
    Observe that the map $\rep_{2c}:\Z\to \Dcal^+\setminus \left( \0\0\Dcal^* \cup
    \1\1\Dcal^* \right)$ is an increasing bijection with respect to the
    order~$\prec$:
\begin{center}
\begin{tikzcd}[column sep=tiny]
     \cdots \arrow[r, phantom, "<"]
    &-5 \arrow[d, mapsto] \arrow[r, phantom, "<"]
    &-4 \arrow[d, mapsto] \arrow[r, phantom, "<"]
    &-3 \arrow[d, mapsto] \arrow[r, phantom, "<"]
    &-2 \arrow[d, mapsto] \arrow[r, phantom, "<"]
    &-1 \arrow[d, mapsto] \arrow[r, phantom, "<"]
    & 0 \arrow[d, mapsto] \arrow[r, phantom, "<"]
    & 1 \arrow[d, mapsto] \arrow[r, phantom, "<"]
    & 2 \arrow[d, mapsto] \arrow[r, phantom, "<"]
    & 3 \arrow[d, mapsto] \arrow[r, phantom, "<"]
    & 4 \arrow[d, mapsto] \arrow[r, phantom, "<"]
    & \cdots \\
      \cdots   \arrow[r, phantom, "\prec"]
    & \1\0\1\1 \arrow[r, phantom, "\prec"]
    & \1\0\0   \arrow[r, phantom, "\prec"]
    & \1\0\1   \arrow[r, phantom, "\prec"]
    & \1\0     \arrow[r, phantom, "\prec"]
    & \1       \arrow[r, phantom, "\prec"]
    & \0       \arrow[r, phantom, "\prec"]
    & \0\1     \arrow[r, phantom, "\prec"]
    & \0\1\0   \arrow[r, phantom, "\prec"]
    & \0\1\1   \arrow[r, phantom, "\prec"]
    & \0\1\0\0 \arrow[r, phantom, "\prec"]
    & \cdots
\end{tikzcd}
\end{center}

We now show that the two's complement numeration system coincides with
the Dumont--Thomas complement numeration system associated with a two-sided fixed
point of 2-uniform substitution. 

\begin{proposition}\label{prop:rep_equality_2c}
    Let $\psi:A\to A^*$ be some 2-uniform substitution
    and let $\beta\in\Per(\psi)$ be some two-sided periodic point of period 1.
    Then $\rep_{\beta}$ is the two's complement numeration system, that is,
    $\rep_{\beta}=\rep_{2c}$.
\end{proposition}

\begin{proof}
    From Proposition~\ref{prop:rep-u-is-increasing},
    $\rep_{\beta}:\Z\to\{\0,\1\}\Dcal^*$
    is an increasing map with respect to the order $\prec$.
    Thus, it is an increasing bijection $\Z\to\rep_\beta(\Z)$.
    From Lemma~\ref{lem:repu-image}, we have 
    \begin{equation*}
    \begin{aligned}
        \rep_\beta(\Z) &= \bigcup_{\ell\in\N}
        \Lcal_{\ell p+1}(\Acal_{\psi,s}) \setminus \{\0\Wmin, \1\Wmax\} \Dcal^*
        = \Lcal_{\geq 1}(\Acal_{\psi,s}) \setminus \{\0\0,\1\1\}\Dcal^*\\
        &=\Dcal^+\setminus \left( \0\0\Dcal^* \cup \1\1\Dcal^* \right),
    \end{aligned}
    \end{equation*}
    since $p = 1$, $\Lcal(\Acal_{\psi,s}) = \Dcal^*$, $\Wmin = \0$ and $\Wmax = \1$.
    Also, $\rep_{\beta}(0)=\0$.
    On the other hand, the map $\rep_{2c}:\Z\to \Dcal^+\setminus \left( \0\0\Dcal^* \cup
    \1\1\Dcal^* \right)$ is an increasing bijection with respect to the
    order~$\prec$ and $\rep_{2c}(0)=\0$.
    From Theorem~\ref{thm:dumont-thomas-characterization-increasing},
    we conclude $\rep_{\beta}=\rep_{2c}$.
\end{proof}

Note that the Thue-Morse substitution has no two-sided fixed point, so the above
result does not hold for numeration systems based on fixed points of the
Thue-Morse substitution; see Table~\ref{tab:num_systems}.

\subsection{Fibonacci analogue of the two's complement numeration system}

In what follows, the Fibonacci sequence
$(F_n)_{n\geq 0}$,
$F_{n} = F_{n-1} + F_{n-2}$, for all $n \geq 2$,
is defined with the initial conditions
$ F_0 = 1$, $F_1 = 2$.
We denote $\Dcal = \{\0,\1\}$.

In \cite{zbMATH07799058}, a Fibonacci analogue of the two's complement numeration
system for nonnegative and negative integers was defined
from the value map $\valFc:\Dcal^*\to\Z$ 
by $\valFc(w) = \sum_{i=0}^{k-1}w_i F_{i} - w_{k-1}F_{k}$
for every binary word $w=w_{k-1}\cdots w_0\in\Dcal^{k}$.
It is an analog of the two's complement value map $\valTwoC$, 
using Fibonacci numbers instead of powers of 2.
It was proved in \cite{zbMATH07799058} that
for every $n\in\Z$ there exists a
unique odd-length word 
$w\in L=\Dcal(\Dcal\Dcal)^*\setminus
    \left( \Dcal^*11\Dcal^* \cup 000\Dcal^* \cup 101\Dcal^*\right)$ 
such that $n=\valFc(w)$.
It defines the map
$\repFc:\Z\to L$ by the rule $n\mapsto w$.


We show that the Dumont--Thomas complement numeration system obtained from the two-sided Fibonacci
word is the Fibonacci analogue of the two's complement numeration system
introduced in~\cite{zbMATH07799058}.

\begin{proposition}\label{prop:rep_equality}
	Let $\varphi:a\mapsto ab,b\mapsto a$ be the Fibonacci substitution
    and let $\gamma\in\Per(\varphi)$ be the periodic point of period 2 with
    seed $s = b|a$.
    Then $\rep_{\gamma}$ is the Fibonacci analogue of the two's complement
    numeration system, that is, $\rep_{\gamma}=\repFc$.
\end{proposition}

\begin{proof}
    From Proposition~\ref{prop:rep-u-is-increasing},
    $\rep_{\gamma}:\Z\to\{\0,\1\}\Dcal^*$
    is an increasing map with respect to the order $\prec$.
    Thus, it is an increasing bijection $\Z\to\rep_\gamma(\Z)$.
    Denote $L = \Dcal(\Dcal\Dcal)^*\setminus
	\left( \Dcal^*\1\1\Dcal^* \cup\0\0\0\Dcal^* \cup \1\0\1\Dcal^*\right)$.
    From Lemma~\ref{lem:repu-image}, we have 
    \begin{equation*}
    \begin{aligned}
        \rep_\gamma(\Z) &= \bigcup_{\ell\in\N}
        \Lcal_{\ell p+1}(\Acal_{\varphi,s}) \setminus \{\0\Wmin, \1\Wmax\} \Dcal^*\\
        &= (\Dcal(\Dcal\Dcal)^* \setminus \Dcal^*\1\1\Dcal^*) 
        \setminus \{\0\0\0, \1\0\1\} \Dcal^* = L,\\
    \end{aligned}
    \end{equation*}
    since $p = 2$,
    $\Lcal(\Acal_{\varphi,s}) = \Dcal^*\setminus \Dcal^*\1\1\Dcal^*$, 
    $\Wmin = \0\0$ and $\Wmax = \0\1$.
    From \cite{zbMATH07799058}, the map $\repFc$ is an increasing bijection $\Z \to
    L$ with respect to the order $\prec$.
    Moreover,
    $\rep_{\gamma}(0) = \0 = \rep_{\Fcal c}(0)$.
    From Theorem~\ref{thm:dumont-thomas-characterization-increasing},
    $\repFc = \rep_{\gamma}$.
\end{proof}

We leave open the following question.

\begin{question}
    If $\rep_{\gamma}$ is the Fibonacci analogue of the two's complement
    numeration system, then what is the meaning of $\rep_{\delta}$?
    Can we define it from some value map?
    Recall that $\delta\in\Per(\varphi)$ is the periodic point of period 2
    of the Fibonacci substitution $\varphi$ with seed $a|a$, see
    Table~\ref{tab:num_systems}.
\end{question}

\subsection*{Acknowledgments}
The authors would like to thank the reviewer for their valuable comments
leading in particular to a improved buildup in the introduction and clearer
notation for the concatenation of integers (not to be misinterpreted as
multiplication).
This work was partially funded in France by
ANR CODYS (ANR-18-CE40-0007) and ANR IZES (ANR-22-CE40-0011).
The second author acknowledges financial support by the Barrande fellowship program
and the Grant Agency of Czech Technical University in Prague
(SGS20/183/OHK4/3T/14).

\bibliographystyle{myalpha}

\bibliography{biblio}

\end{document}